\documentclass{amsart}
\usepackage{amsmath,amssymb,amsthm, tikz, enumerate}
\usepackage{xcolor}

\newtheorem{theorem}{Theorem}
\newtheorem{lemma}[theorem]{Lemma}
\newtheorem{corollary}[theorem]{Corollary}

\newtheorem{proposition}[theorem]{Proposition}

\theoremstyle{definition}
\newtheorem{definition}[theorem]{Definition}

\newcommand{\co}{\colon\thinspace}

\newcommand{\abs}[1]{\left\lvert #1 \right\rvert}

\DeclareMathOperator{\Area}{Area}
\DeclareMathOperator{\Vol}{Vol}

\begin{document}

\title[Simplicial volume and $Z_0$]{Simplicial volume and $0$-strata of \\separating filtrations}
\author{Hannah Alpert}
\address{Auburn University, 221 Parker Hall, Auburn, AL 36849}
\email{hcalpert@auburn.edu}


\begin{abstract}
We use Papasoglu's method of area-minimizing separating sets to give an alternative proof, and explicit constants, for the following theorem of Guth and Braun--Sauer: If $M$ is a closed, oriented, connected $n$-dimensional manifold, with a Riemannian metric such that every ball of radius $1$ in the universal cover of $M$ has volume at most $V_1$, then the simplicial volume of $M$ is at most the volume of $M$ times a constant depending on $n$ and $V_1$.
\end{abstract}

\maketitle

\section{Introduction}

The purpose of this paper is to prove the following theorem.

\begin{theorem}\label{thm:big}
Let $M$ be a closed, oriented, connected $n$-dimensional Riemannian manifold, and let $\Gamma = \pi_1(M)$.  Suppose that for all points $\widetilde{p}$ in the universal cover $\widetilde{M}$ of $M$, we have $\Vol B(\widetilde{p}, 1) \leq V_1$.  Then
\[\Vert M \Vert_{\Delta} \leq 16^n (n!)^2 \cdot V_1 \cdot \Vol M,\]
where $\Vert M \Vert_{\Delta}$ denotes the Gromov simplicial volume of $M$.  Furthermore, if $V_1 < \frac{1}{n!}$, then the image of the fundamental homology class of $M$ under the classifying map is zero in $H_*(B\Gamma; \mathbb{Q})$, so $\Vert M \Vert_{\Delta} = 0$.
\end{theorem}

Only the constants $16^n (n!)^2 \cdot V_1$ and $\frac{1}{n!}$ are new.  The theorem, with non-explicit constants, is proved by Guth in~\cite{guth11} for the case where $M$ admits a hyperbolic metric; the proof applies to any manifold with residually finite fundamental group.  For the case where $V_1$ is close to zero, Guth adapts the same proof to show in~\cite{guth17} that $M$ has bounded Urysohn $(n-1)$-width; that is, $M$ admits a map to an $(n-1)$-dimensional simplicial complex, for which all fibers have diameter at most $2$.  Liokumovich, Lishak, Nabutovsky, and Rotman in~\cite{liokumovich2022filling} generalize this Urysohn width theorem to the case where $M$ is not necessarily a manifold, and Papasoglu in~\cite{papasoglu20} proves the same statement by a shorter method, similar to the method of minimal hypersurfaces which Guth in~\cite{guth10} adapts from Schoen and Yau's papers~\cite{schoen78, schoen79}.  Braun and Sauer in~\cite{braun21} adapt Guth's original proof to generalize it to the case where the fundamental group of $M$ is not necessarily residually finite.  In~\cite{sabourau22}, Sabourau proves a related result, that if the volume of $M$ is sufficiently small, then in the universal cover there are balls of all radii greater than $1$ with larger-than-hyperbolic volume.

Braun and Sauer speculate about whether the method of Papasoglu can be used to prove their theorems in a shorter way.  Here we give an affirmative answer: the method of Cantor bundles from~\cite{braun21} can serve the same role of removing the assumption that $\pi_1(M)$ is residually finite, while the method of Papasoglu replaces the more complicated method of Guth.  Because Papasoglu's method is so much simpler, it allows us to give explicit constants in the theorem statement.

Braun and Sauer in~\cite{braun21} also consider the extent to which the arguments about the Gromov simplicial norm, which involves homology with real or rational coefficients, can be adapted to integer coefficients.  In Theorem~\ref{thm:cor-source} and Corollaries~\ref{cor:first} and~\ref{cor:second} we recover their results except for an additional hypothesis that $\pi_1(M)$ should be torsion-free, and we give explicit constants.  Very roughly, the change to integer coefficients costs a factor of $(n+1)!$, using a method from Campagnolo and Sauer in~\cite{campagnolo19}.  Corollary~\ref{cor:first} shows that if $M$ is aspherical, then its integral foliated simplicial volume is at most $(n+1)16^n(n!)^3 \cdot V_1 \cdot \Vol M$, and Corollary~\ref{cor:second} gives the resulting upper bounds on the $\ell^2$-Betti numbers and the Euler characteristic.

Section~\ref{sec:background} gives the properties of simplicial volume that we need to prove Theorem~\ref{thm:big}.  In Section~\ref{sec:baby} we prove a version of Theorem~\ref{thm:big} with the extra assumption that every nontrivial loop in $M$ has length greater than $2$, as a warm-up for proving Theorem~\ref{thm:big} in Section~\ref{sec:proof}.  Section~\ref{sec:foliated} gives the results on the integral foliated simplicial volume.

\emph{Acknowledgments.} Thank you to Alexey Balitskiy and Caterina Campagnolo for helpful conversations about the ideas in this paper, as well as to an anonymous referee whose careful reading and suggestions caused many proof holes to be identified and fixed.  The author is supported in part by the Simons Foundation (Gift Number 965348).

\section{Preliminaries on simplicial volume}\label{sec:background}

This section includes some background information on simplicial volume that links it to the main part of the proof, which is about separating filtrations.  The first subsection includes the definition of simplicial volume and a theorem bounding the simplicial volume in terms of the number of rainbow simplices in a vertex-colored cycle.  The second subsection includes the definition of a separating filtration and a lemma bounding the number of rainbow simplices in terms of the number of points in the $0$-dimensional stratum of a separating filtration.

\subsection{Simplex straightening}

Let $z = \sum_i a_i \sigma_i$ be a singular $d$-cycle on a space $P$ (say, a CW complex) with real coefficients $a_i \in \mathbb{R}$ and simplices $\sigma_i \co \Delta^d \rightarrow P$.  The $L^1$ norm of $z$, denoted by $\abs{z}_1$, is $\sum_i \abs{a_i}$, and the \textit{\textbf{simplicial norm}} of a given homology class is the infimum of $\abs{z}_1$ over all cycles $z$ representing the homology class.  The simplicial norm of the class $[z]$ is denoted by $\Vert [z] \Vert_{\Delta}$.  The \textit{\textbf{simplicial volume}} of a closed, oriented, connected manifold $M$, denoted by $\Vert M \Vert_{\Delta}$, is the simplicial norm of the fundamental homology class of $M$ with $\mathbb{R}$-coefficients.  The simplicial norm was introduced by Gromov in~\cite{gromov82}.  

One foundational property of simplicial norm from~\cite[Sections 2.3, 3.3]{gromov82} (or see~\cite[Theorem 4.1]{ivanov87}) is that if $f \co P \rightarrow Q$ is a continuous map between path-connected spaces $P$ and $Q$ that induces an isomorphism between their fundamental groups, then the induced map on homology $f_* \co H_*(P) \rightarrow H_*(Q)$ preserves simplicial norm.  Thus, in particular, the simplicial volume of a manifold $M$ with fundamental group $\Gamma$ is equal to the simplicial norm of the image of $[M]$ in the homology of the classifying space $B\Gamma$.

Our use of simplicial norm in this paper is based on the following theorem, a special case of the Amenable Reduction Lemma from~\cite{gromov09} (or see~\cite{alpert16}).  Given a singular cycle $z$ on a space $P$, we can consider the domain spaces of all of its simplices, and imagine identifying any of their faces that cancel in $\partial z$; because $\partial z = 0$, every face is identified with at least one other face.  The resulting simplicial complex---or strictly speaking a $\Delta$-complex because some faces of the same simplex may be identified---has a simplicial cycle in which the coefficient of each simplex of the $\Delta$-complex equals the coefficient in $z$ of the corresponding singular simplex.  We can think of $z$ as the image of this simplicial cycle under the map from the $\Delta$-complex into $P$, and we can talk about the vertices and edges of $z$, meaning the vertices and edges of the $\Delta$-complex, equipped with their maps into $P$.  If $P$ is path-connected, we define a \textbf{\textit{$\pi_1$-killing vertex coloring}} of $z$ to be a way to assign colors to the vertices of $z$ such that for each color, if we take the union of all edges of $z$ for which both vertices are that color, then the map of this $1$-complex into $P$ induces the zero map on $\pi_1$.  (In particular, if $z$ does not contain any edges from a vertex to itself, then coloring every vertex a different color is $\pi_1$-killing.)  We define a \textbf{\textit{rainbow simplex}} of such a coloring to be any simplex in $z$ for which all $d+1$ vertices are different colors.

\begin{theorem}[\cite{gromov09}]\label{thm:rainbow}
Let $z = \sum_i a_i\sigma_i$ be a singular $d$-cycle on a path-connected CW complex $P$, with a $\pi_1$-killing vertex coloring.  Then the simplicial norm of the homology class of $z$ satisfies
\[\Vert [z] \Vert_{\Delta} \leq \sum_{\mathrm{rainbow}\ \sigma_i} \abs{a_i}.\]
\end{theorem}

\begin{proof}[Proof sketch]
On any space with contractible universal cover, such as the classifying space $B\Gamma$ where $\Gamma = \pi_1(P)$, we can define a notion of simplex straightening: for each $(d+1)$-tuple of points in the universal cover $E\Gamma$, the idea is to make a choice of $d$-simplex with those vertices, in a way that agrees with taking faces, translating by $\Gamma$, and permuting the vertices (with sign).  More formally, instead of literally choosing a single $d$-simplex, for the permutation invariance we need to choose the signed average of all of its permutations, and for the $\Gamma$-invariance, in the case where $\Gamma$ has finite-order elements we need to choose the average of any $\Gamma$-translations of our simplex that permute its vertices.  (It is not possible to ensure the $\Gamma$-invariance through cleverly symmetric choices of simplex, because each $\Gamma$-translation of a simplex has a different barycenter due to the freeness of the $\Gamma$-action, yet each permutation fixes the barycenter.)  Every cycle is homologous to its straightened version, and if a given simplex lifts to the universal cover in such a way that two of its vertices are the same, its straightening is zero because of the signed permutation invariance.

Thus, let $\alpha \co P \rightarrow B\Gamma$ be the classifying map.  The classes $[z]$ and $\alpha_*[z]$ have the same simplicial norm.  We homotope $\alpha(z)$ so that all vertices of each color and all edges among them go to a single point---this is possible because the coloring is $\pi_1$-killing.  Then we straighten.  The result is homologous to $\alpha(z)$, and each rainbow simplex contributes the same amount to the simplicial norm as it did in $[z]$, but each non-rainbow simplex becomes zero, because its two vertices of the same color lift to the same point in the universal cover.
\end{proof}

A special case of the theorem above is when there are no rainbow simplices.  In this case, we can conclude that the simplicial norm is zero, but we can also say something a bit stronger: the image of the homology class in the classifying space is torsion.

\begin{corollary}\label{thm:rainbow-cor}
Let $P$ be a path-connected CW complex, let $\Gamma = \pi_1(P)$, and let $\alpha \co P \rightarrow B\Gamma$ be the classifying map.  Let $z$ be a singular $d$-cycle on $P$, homologous in $H_*(P; \mathbb{R})$ to an element of $H_*(P; \mathbb{Z})$, and suppose that $z$ admits a $\pi_1$-killing vertex coloring that has no rainbow simplices.  Then the class $\alpha_*[z]$ is zero in $H_*(B\Gamma; \mathbb{Q})$.
\end{corollary}

\begin{proof}[Proof sketch]
In the proof above, the straightening of the homotoped cycle $\alpha(z)$ is zero, because there are no rainbow simplices.  Thus, $\alpha_*[z]$ is zero in $H_*(B\Gamma; \mathbb{R})$.  But the change of coefficients from $H_*(B\Gamma; \mathbb{Q})$ to $H_*(B\Gamma; \mathbb{R})$ is just the tensor product with $\mathbb{R}$, which gives an injection, so $\alpha_*[z]$ is zero in $H_*(B\Gamma; \mathbb{Q})$ as well.  Alternatively, the simplex straightening in Theorem~\ref{thm:rainbow} is well-defined with coefficients in $\mathbb{Q}$ rather than in $\mathbb{R}$.
\end{proof}

\subsection{Triangulating a separating filtration}

The next lemma links the idea of counting rainbow simplices with Papasoglu's method of area-minimizing separating sets in~\cite{papasoglu20}.  Papasoglu's method is to find a filtration 
\[M = Z_n \supseteq Z_{n-1} \supseteq \cdots \supseteq Z_1 \supseteq Z_0,\]
such that each $Z_i$ is an $i$-dimensional set, minimizing $i$-dimensional area subject to the condition that every connected component of $Z_{i+1} \setminus Z_i$ is contained in a ball of some radius $R$ in $M$.  We define these sets in terms of Riemannian polyhedra as in~\cite{nabutovsky19}.  Informally, we need Riemannian polyhedra to be a class of spaces that includes closed manifolds and allows cutting and gluing as in Lemma~\ref{lem:replacement}.  More formally, we refer to a closed union of faces (that is, a subcomplex) of a standard Euclidean simplex as a polyhedron.  Note that a polyhedron is required to have only finitely many faces.  A pure $d$-dimensional polyhedron has the property that every maximal face has dimension $d$.  A function on a polyhedron is defined to be smooth if its restriction to the closure of every face is smooth.  We define a \textbf{\textit{Riemannian $d$-polyhedron}} to be a pure $d$-dimensional polyhedron with a Riemannian metric on every maximal simplex, such that the metrics agree on common faces.  In our setting, the Riemannian polyhedra $Z_i$ are embedded in the Riemannian manifold $M$, so their Riemannian metrics are induced from $M$.  If $P$ is a pure $d$-dimensional Riemannian polyhedron, we denote its $d$-dimensional volume by $\Area_d(P)$.

A \textbf{\textit{codimension $1$ subpolyhedron}} of a Riemannian $d$-polyhedron $P$ is a pure $(d-1)$-dimensional polyhedron $Z$ that is a subcomplex of some smooth triangulation of $P$ and satisfies the \textbf{\textit{smooth nesting condition}}: the relative interior of every face of $Z$ is contained in the relative interior of a face of $P$ of strictly greater dimension.  By a smooth triangulation (or refinement) of $P$ we mean a polyhedron diffeomorphic to $P$, for which every closed simplex of $P$ is a subcomplex.  When checking the smooth nesting condition, we use the original polyhedron structure of $P$, rather than the refinement; the polyhedron structure on $Z$ is the one arising from the subcomplex of the refinement.
The subpolyhedron $Z$ is a Riemannian polyhedron with the Riemannian metric induced from $P$.

Given any $R > 0$ we say that $Z$ is \textbf{\textit{$R$-separating}} in $P$ if $Z$ is a codimension $1$ subpolyhedron of $P$ such that every connected component of $P \setminus Z$ is contained in a ball of radius $R$.  Because $Z$ is pure $(d-1)$-dimensional, we know that $\Area_{d-1}(Z)$ is well-defined; later, in Lemmas~\ref{lem:replacement} and~\ref{lem:volball} we will be choosing $Z$ to be nearly area-minimizing among $R$-separating subpolyhedra of $P$.

\begin{figure}
\begin{center}
\begin{tikzpicture}
\draw (0, 2)--(-2, 2)--(-2, -2)--(0, -2)-- (1.00000000000000, -0.267949192431123) -- (1.00000000000000, 3.73205080756888) --cycle;
\draw[very thick] (-2, 0.333333333333333) -- (0, 1.00000000000000) -- (0.750000000000000, 2.79903810567666) ;
\draw[very thick] (0.500000000000000, 2.86602540378444) -- (1.00000000000000, 2.73205080756888) ;
\draw[fill=white] (0, 2)-- (1.73205080756888, 1.00000000000000) -- (1.73205080756888, -3.00000000000000) --(0, -2)--cycle;
\draw[very thick] (-1.50000000000000, 0.500000000000000) -- (0, -0.700000000000000) -- (0.866025403784439, -2.00000000000000) -- (0.866025403784439, 0.833333333333333) -- (0, 1.00000000000000) ;
\draw[very thick] (0.866025403784439, 0.833333333333333) -- (1.29903810567666, 1.25000000000000) (1.40729128114971, -2.81250000000000) -- (0.866025403784439, -2.00000000000000) -- (0.866025403784439, -2.50000000000000) ;
\draw[fill=gray!40] (-0.333333333333333, 0.888888888888889) circle (.07);
\draw[fill=gray!40] (-0.666666666666667, -0.166666666666667) circle (.07);
\draw[fill=gray!40] (0.866025403784439, 5.55111512312578e-17) circle (.07);
\draw[fill=gray!40] (0.500000000000000, 2.19935873711777) circle (.07);
\end{tikzpicture}
\end{center}
\caption{In an $R$-separating filtration, $Z_0$ (shown as gray dots) separates $Z_1$ (shown as thick lines) into pieces of size at most $R$, and $Z_1$ separates $Z_2$ (shown as three planes) into pieces of size at most $R$.  The smooth nesting condition implies that $Z_0$ avoids the points where multiple edges of $Z_1$ come together, and $Z_1$ does not run along the line where multiple planes of $Z_2$ come together.}\label{fig:rsep}
\end{figure}
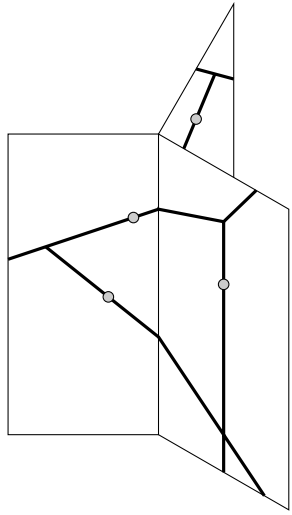

We define an \textbf{\textit{$R$-separating filtration}} of $M$ to consist of Riemannian polyhedra
\[M = Z_n \supseteq Z_{n-1} \supseteq \cdots \supseteq Z_1 \supseteq Z_0,\]
such that each $Z_i$ is an $R$-separating subpolyhedron of $Z_{i+1}$.  A neighborhood in an example of an $R$-separating filtration is depicted in Figure~\ref{fig:rsep}.  The smooth nesting condition implies that every point of $Z_0$ has a neighborhood with a diffeomorphism that sends our filtration to the filtration
\[\mathbb{R}^n \supseteq \mathbb{R}^{n-1} \supseteq \cdots \supseteq \mathbb{R}^1 \supseteq \mathbb{R}^0.\]
Thus, the following lemma shows that if we triangulate consistent with the filtration, we can produce a $\pi_1$-killing coloring with a controlled number of rainbow simplices.

\begin{lemma}\label{lem:triangulate}
Let $M$ be a closed, connected $n$-dimensional Riemannian manifold with an $R$-separating filtration 
\[M = Z_n \supseteq Z_{n-1} \supseteq \cdots \supseteq Z_1 \supseteq Z_0.\]
Suppose that for every ball of radius $R$ in $M$, the map on $\pi_1$ induced by its inclusion into $M$ is the zero map.  Then there is a triangulation of $M$ with a $\pi_1$-killing coloring, such that the number of rainbow simplices is $2^n \cdot \#Z_0$.
\end{lemma}

\begin{proof}
We want a triangulation of $M$ such that each $Z_i$ is a subcomplex of the triangulation.  To do this, we show by induction on $i$ that there exists a triangulation of $Z_i$ such that every $Z_{i'}$ for $i' < i$ is a subcomplex of the triangulation.  By the subpolyhedron property, there is a refinement of $Z_i$ that contains $Z_{i-1}$ as a subcomplex, and by the inductive hypothesis, there is a refinement of $Z_{i-1}$ that contains all lower levels $Z_{i'}$ as subcomplexes.  The resulting partition of $Z_i$ into faces of the refinements of $Z_{i-1}$ and of $Z_i \setminus Z_{i-1}$ forms a Whitney stratified space, so there is a smooth triangulation of $Z_i$ that refines this stratification.  (See~\cite{goresky88} for exposition on Whitney stratified spaces, and~\cite{goresky78} for how to construct a stratified triangulation.)  Continuing this process up through dimension $n$ we obtain the desired triangulation of $M$.

We color all the points in $M$, such that two points are the same color if and only if they are in the same level $Z_{i} \setminus Z_{i-1}$ (for $i = 0, \ldots, n$) and they are in the same connected component of $Z_i \setminus Z_{i-1}$.  Our initial triangulation of $M$ does not necessarily have the number of rainbow simplices that we want, but we claim that its barycentric subdivision does.  We observe that in the initial triangulation, the relative interior of each simplex has only one color, and that if one simplex is a face of another, the two simplices don't come from different components of $Z_i \setminus Z_{i-1}$ for the same $i$---either they are the same color, or they are at different levels of the filtration.  Thus, the corresponding property is true of the barycentric subdivision: if two vertices are in the same simplex, then either they are the same color, or they are at different levels.  This implies that if a simplex is rainbow, then all of its vertices are at different levels, and in an $n$-simplex, this means that among its $n+1$ vertices there must be exactly one at each level $Z_0, Z_1 \setminus Z_0, \ldots, Z_n \setminus Z_{n-1}$.

Recall that the smooth nesting condition implies that every point in $Z_0$ has a neighborhood in $M$ and a diffeomorphism from this neighborhood to $\mathbb{R}^n$ with the property that each level $Z_i$ gets sent to $\mathbb{R}^i \subseteq \mathbb{R}^n$. This implies that in the barycentric subdivision, there are exactly $2^n$ rainbow simplices containing each point of $Z_0$: from the point in $Z_0$, there are two directions in $Z_1$, corresponding to the two components of $\mathbb{R}^1 \setminus \mathbb{R}^0$; for each of these directions, there are two directions in $Z_2$, corresponding to the two components of $\mathbb{R}^2 \setminus \mathbb{R}^1$, and so on.  Each such chain of choices of directions corresponds to exactly one simplex.  Because the filtration is $R$-separating, each color (and all the edges among vertices of that color) is contained in a ball of radius $R$, which contributes nothing to $\pi_1(M)$, so the coloring is $\pi_1$-killing.
\end{proof}

\section{Large-systole case}\label{sec:baby}

The purpose of this section is to prove the following weaker version of Theorem~\ref{thm:big}, in preparation for proving the full version in the next section.  For any Riemannian manifold $M$, let $V_1(M)$ denote the supremal volume of a ball of radius $1$ in $M$.

\begin{theorem}\label{thm:baby}
Let $M$ be a closed, oriented, connected $n$-dimensional Riemannian manifold, and let $\Gamma = \pi_1(M)$.  Suppose that every homotopically nontrivial loop in $M$ has length greater than $2$.  Then
\[\Vert M \Vert_{\Delta} \leq 16^n (n!)^2 \cdot V_1(M) \cdot \Vol M.\]
Furthermore, if $V_1(M) < \frac{1}{n!}$, then the image of the fundamental homology class of $M$ under the classifying map is zero in $H_*(B\Gamma; \mathbb{Q})$, so $\Vert M \Vert_{\Delta} = 0$.
\end{theorem}

The strategy is to select a $1$-separating filtration that is (nearly) area-minimizing at each level, and to bound the number of points in the $0$-dimensional level $Z_0$ in terms of $\Vol M$.  Then we can apply the statements from the previous section to relate the simplicial volume to the number of points in $Z_0$.  The next two lemmas show that every point in $Z_0$ is in a ball of fairly large volume.

\begin{lemma}\label{lem:replacement}
Let $M$ be a closed, oriented, connected $n$-dimensional Riemannian manifold, and let $R > 0$ be arbitrary.  Let $P$ be a pure $d$-dimensional Riemannian polyhedron embedded in $M$, and let $Z$ be an $R$-separating subpolyhedron of $P$.  Suppose that $\Area_{d-1}(Z)$ is within $\varepsilon$ of the infimal area of $R$-separating subpolyhedra of $P$.  Then for all $p \in M$ and all $r_1, r_2$ with $0 < r_1 < r_2 < R$ we have
\[\int_{r_1}^{r_2} \Area_{d-1} (Z \cap B(p, \rho))\ d\rho \leq \Area_d (P \cap B(p, r_2) \setminus B(p, r_1)) + 2\varepsilon R.\]
\end{lemma}

\begin{proof}
Given $Z, p, r_1, r_2$, we approximate the distance function on $M$ by a smooth function that is within a small distance $\delta$ of $M$ and is $(1+\delta)$-Lipschitz.  If $\widehat{B}(p, r)$ denotes the ball of radius $r$ around $p$ computed according to this approximate distance function, then $B(p, r-\delta) \subseteq \widehat{B}(p, r) \subseteq B(p, r+\delta)$.  We choose $\delta$ small enough that $r_2 + 2\delta < R$ and $3\delta \Area_{d-1}(Z \cap B(p, R)) + \delta \Area_d(P \cap B(p, R)) < \varepsilon R$.

For almost all $\rho$, the function $\Area_d(P \cap \widehat{B}(p, \rho))$ is differentiable at $\rho$.  The coarea inequality (see~\cite[Theorem 13.4.2]{burago88}) says that because the approximate-distance-to-$p$ function on $P$ is $(1+\delta)$-Lipschitz with fibers $P \cap \widehat{S}(p, \rho)$, where $\widehat{S}(p, \rho)$ denotes $\partial \widehat{B}(p, \rho)$, we have
\[\Area_{d-1}(P \cap \widehat{S}(p, \rho)) \leq (1+\delta) \frac{d}{d\rho} \Area_d(P \cap \widehat{B}(p, \rho)).\]


On the other hand, we claim that for almost all $\rho \in [0, r_2-\delta]$, replacing $Z \cap \widehat{B}(p, \rho)$ in $Z$ by $P \cap \widehat{S}(p, \rho)$ gives another $R$-separating subpolyhedron $Z'$ of $P$.  To see this, we consider every simplex of the refinement of $P$ that contains $Z$ as a subcomplex, and select $\rho$ to be a regular value of the smooth approximate-distance-to-$p$ function when restricted to all of these (finitely many) simplices.  Geometrically, this means that $\widehat{S}(p, \rho)$ is transverse to every simplex of the refinement of $P$.  We view $P$ as a Whitney stratified space in the following way: take the strata to be the faces of this refinement, and then subdivide them according to membership in $\widehat{S}(p, \rho)$.  The closure of each stratum is a manifold with corners, so the stratification satisfies the Whitney properties, and thus there is a smooth triangulation of $P$ that refines this stratification~\cite{goresky88, goresky78}.  Then $Z'$ is a subcomplex of this triangulation.  It satisfies the smooth nesting property because every original face of $P$ intersects $\widehat{S}(p, \rho)$ in a set of strictly lesser dimension.  The set $Z'$ is $R$-separating in $P$ because each connected component of $P \setminus Z'$ is contained either in $\widehat{B}(p, R)$ or in a connected component of $P \setminus Z$.

Thus, because $Z$ is area-minimizing up to $\varepsilon$ we have
\[\Area_{d-1}(Z \cap \widehat{B}(p, \rho)) \leq \Area_{d-1} (P \cap \widehat{S}(p, \rho)) + \varepsilon.\]
Integrating as $\rho$ varies between $r_1+\delta$ and $r_2-\delta$, and observing that $\varepsilon(r_2 - r_1-2\delta) \leq \varepsilon R$, gives
\[\int_{r_1 + \delta}^{r_2 - \delta} \Area_{d-1} (Z \cap \widehat{B}(p, \rho))\ d\rho \leq (1+\delta) \Area_d(P \cap \widehat{B}(p, r_2-\delta) \setminus \widehat{B}(p, r_1+\delta)) + \varepsilon R,\]
and so, plugging in our choice of $\delta$, we have
\[\int_{r_1}^{r_2} \Area_{d-1} (Z \cap B(p, \rho))\ d\rho \leq \Area_d (P \cap B(p, r_2) \setminus B(p, r_1)) + 2\varepsilon R.\]
\end{proof}

Applying the lemma above, along with repeated integration, then gives the following bound on the volumes of balls around points of $Z_0$.

\begin{lemma}\label{lem:volball}
Let $M$ be a closed $n$-dimensional Riemannian manifold.  For all $R > 0$ and all $\varepsilon > 0$, there exists an $R$-separating filtration
\[M = Z_n \supseteq Z_{n-1} \supseteq \cdots \supseteq Z_1 \supseteq Z_0,\]
such that for all $p \in M$ and all $r_1, r_2$ with $0 < r_1 < r_2 < R$ we have
\[\Area_0(Z_0 \cap B(p, r_1)) \cdot \frac{(r_2 - r_1)^n}{n!} \leq \Vol B(p, r_2) + \varepsilon.\]
\end{lemma}

\begin{proof}
By induction on $i$ we can prove the following statement: if each $Z_i$ is area-minimizing up to $\varepsilon_i$, then for all $0 < r_1 < r_2 < R$, we have
\begin{align*}
\Area_0(Z_0 \cap B(p, r_1)) \cdot \frac{(r_2 - r_1)^i}{i!} \leq \Area_i(Z_i \cap B(p, r_2)) + \\
+ 2\varepsilon_0R^i + 2\varepsilon_1 R^{i-1} + \cdots + 2\varepsilon_{i-1}R.
\end{align*}
The base case $i = 0$ says $\Area_0(Z_0 \cap B(p, r_1)) \leq \Area_0(Z_0 \cap B(p, r_2))$, which is true.  The inductive step is obtained by replacing $r_2$ by $\rho$ in the inductive hypothesis, integrating as $\rho$ ranges from $r_1$ to $r_2$, and applying Lemma~\ref{lem:replacement} to the right-hand side of the inequality.

For $i = 0, \ldots, n-1$ we select $\displaystyle\varepsilon_i = \frac{\varepsilon}{2nR^{n-i}}$, so that
\[2\varepsilon_0R^{n} + 2\varepsilon_1 R^{n-1} + \cdots + 2\varepsilon_{n-1}R = \varepsilon.\]
Then plugging $i = n$ into the induction claim gives the desired inequality.
\end{proof}

With this lemma we can prove the special case of the main theorem.

\begin{proof}[Proof of Theorem~\ref{thm:baby}]
First we prove the statement about the case $V_1(M) < \frac{1}{n!}$.  In this case, we let $\varepsilon < \frac{1}{n!} - V_1(M)$, apply Lemma~\ref{lem:volball} with $R = 1$, and apply Lemma~\ref{lem:triangulate} to get a triangulation of $M$ with $2^n$ rainbow simplices for each point in $Z_0$.  For any point $p$ in $Z_0$, taking $r_1 \rightarrow 0$ and $r_2 \rightarrow 1$ in Lemma~\ref{lem:volball} gives a contradiction.  Thus there are no points in $Z_0$, and thus no rainbow simplices.  Corollary~\ref{thm:rainbow-cor} shows that the image of the fundamental class of $M$ is zero in $H_*(B\Gamma; \mathbb{Q})$.

In the case where $V_1(M) \geq \frac{1}{n!}$, for each $\varepsilon > 0$, we apply Lemma~\ref{lem:volball} with $R = 1$, and apply Lemma~\ref{lem:triangulate} to get a triangulation of $M$ with $2^n$ rainbow simplices for each point in $Z_0$.  Letting $V_1 = V_1(M)$ we take a maximal collection of disjoint balls of radius $\frac{1}{4}$, centered at points of $Z_0$.  The concentric balls of radius $\frac{1}{2}$ cover all of $Z_0$.  Let $B_1(\frac{1}{4}), \ldots, B_k(\frac{1}{4})$ denote the disjoint balls of radius $\frac{1}{4}$, and let $B_1(\frac{1}{2}), \ldots, B_k(\frac{1}{2})$ denote the $Z_0$-covering balls of radius $\frac{1}{2}$.  We apply Theorem~\ref{thm:rainbow}, and take the conclusion of Lemma~\ref{lem:volball} first with $r_1 = \frac{1}{2}$ and $r_2 \rightarrow 1$, and then with $r_1 \rightarrow 0$ and $r_2 = \frac{1}{4}$, to obtain
\[\Vert M \Vert_{\Delta} \leq 2^n \cdot \#(Z_0) \leq 2^n \cdot \sum_{i = 1}^k \#\left(Z_0 \cap B_i\left(\frac{1}{2}\right)\right) \leq  2^n \cdot \sum_{i = 1}^k n! \cdot 2^n \cdot (V_1 + \varepsilon) \leq \]
\[\leq 4^n \cdot n! \cdot (V_1 + \varepsilon) \cdot \sum_{i = 1}^k 1 \leq 4^n \cdot n! \cdot (V_1 + \varepsilon) \cdot \frac{\Vol M}{\left(\frac{(1/4)^n}{n!} - \varepsilon\right)},\]
and taking $\varepsilon \rightarrow 0$ gives
\[\Vert M \Vert_{\Delta} \leq 16^n (n!)^2 \cdot V_1 \cdot \Vol M.\]
\end{proof}

\section{Main proof}\label{sec:proof}

\subsection{The idea of Cantor bundles}
Before getting into the technical setup of the proof of Theorem~\ref{thm:big}, we begin with an informal overview of the idea of Cantor bundles from~\cite{braun21}, which was developed to extend the proof of Theorem~\ref{thm:baby} from the case where there may be short nontrivial loops but $\pi_1(M)$ is residually finite, to the case where $\pi_1(M)$ is not residually finite.  First we describe what happens when $\pi_1(M)$ is residually finite, and then we describe the Cantor bundle idea, which is closely analogous.

Let $\Gamma$ be the fundamental group of $M$, and suppose that $\Gamma$ is residually finite.  There are finitely many elements of $\Gamma$ with the property that their deck transformations on $\widetilde{M}$ move some points within distance $2$ of themselves.  By residual finiteness, there is a finite-index subgroup of $\Gamma$ that avoids all of these elements (except the identity), corresponding to a finite-sheeted covering space $\widehat{M}$ of $M$ without nontrivial loops of length at most $2$.  If $\widehat{M}$ has $k$ sheets over $M$, we can think of $\widehat{M} \rightarrow M$ as a locally trivial bundle with fiber $\{1, \ldots, k\}$, and we can also think of this bundle as the quotient of $\{1, \ldots, k\} \times \widetilde{M}$ by the action of $\Gamma$.  Here, the action of $\Gamma$ on $\{1, \ldots, k\}$ comes from identifying $\{1, \ldots, k\}$ with the set of cosets $\Gamma / \pi_1(\widehat{M})$.

Instead of looking for a $1$-separating filtration of $M$, we do the same thing on $\widehat{M}$, or equivalently, we do it equivariantly on $\{1, \ldots, k\} \times \widetilde{M}$.  Then, to estimate the simplicial volume of $M$, instead of triangulating $M$ we triangulate $\widehat{M}$ (or triangulate $\{1, \ldots, k\} \times \widetilde{M}$ equivariantly), multiply each simplex by $\frac{1}{k}$, and project to $M$ to get a fundamental cycle for $M$.

To adapt this method to the case where $\Gamma$ is not residually finite, Braun and Sauer replace the finite set $\{1, \ldots, k\}$ by a Cantor set $X$.  The Cantor set admits a free, continuous action of $\Gamma$, as shown by~\cite{hjorth06}.  Thus, our work to find a $1$-separating filtration, and then a triangulation, is done $\Gamma$-equivariantly on $X \times \widetilde{M}$, or equivalently is done on the quotient $X \times_{\Gamma} \widetilde{M}$, which is a locally trivial bundle over $M$ with fiber equal to $X$.

Informally, even though this cover of $M$ or $\widetilde{M}$ now has uncountably many sheets, over any compact subset of $\widetilde{M}$, we still want only finitely many different kinds of sheets.  To guarantee this property, we introduce the following definitions.  We say that a \textbf{\textit{thick set}} is a subset of $X \times \widetilde{M}$ of the form $A \times S$, where $A \subseteq X$ is clopen and $S \subseteq \widetilde{M}$ is bounded.  We say that a thick set has a \textbf{\textit{non-self-intersecting orbit}} if it does not intersect any of its other $\Gamma$-translates.  When a thick set has a non-self-intersecting orbit, we call the union of the thick set and all its translates a \textbf{\textit{thick orbit}}.  We verify in the following proposition that any collection of finitely many thick orbits has the property that there are only finitely many kinds of sheets over any compact subset.

\begin{proposition}\label{prop:finiteness}
Let $M$ be a closed, oriented, connected $n$-dimensional Riemannian manifold.  Let $\Gamma = \pi_1(M)$, and let $\widetilde{M}$ be the universal cover of $M$.  Let $X$ denote the Cantor set, equipped with a free, continuous action of $\Gamma$, and consider the diagonal action of $\Gamma$ on $X \times \widetilde{M}$.  Let $\Gamma(A_1 \times S_1), \ldots, \Gamma(A_r \times S_r)$ be any collection of finitely many thick orbits in $X \times \widetilde{M}$ under the action of $\Gamma$.  Then for every $R > 0$ and every ball $B(\widetilde{p}, R)$ in $\widetilde{M}$, there exists a partition of $X$ into finitely many clopen sets $X_1, \ldots, X_k$, such that within each set $X_j \times B(\widetilde{p}, R)$, every sheet $\{x\} \times B(\widetilde{p}, R)$ has the same pattern of thick orbits.  That is, for every $j \in \{1, \ldots, k\}$, every $x, y \in X_j$, every $\widetilde{q} \in B(\widetilde{p}, R)$, and every $i \in \{1, \ldots, r\}$, we have $(x, \widetilde{q}) \in \Gamma(A_i \times S_i)$ if and only if $(y, \widetilde{q}) \in \Gamma(A_i \times S_i)$.
\end{proposition}

\begin{proof}
Consider all sets $\gamma A_i$, where $\gamma \in \Gamma$ and $A_i \times S_i$ is a thick set in our collection such that $\gamma S_i$ intersects $B(\widetilde{p}, R)$.  There are only finitely many such sets, and they generate an algebra in $X$ under union, intersection, and complement.  There are finitely many minimal sets in this algebra, all of which are clopen, and we choose $X_1, \ldots, X_k$ to be these sets.  Thus, if $x$ and $y$ are in the same one of these sets, then $x$ and $y$ are in exactly the same sets $\gamma A_i$, so for any $\widetilde{q} \in B(\widetilde{p}, R)$, the points $(x, \widetilde{q})$ and $(y, \widetilde{q})$ are in exactly the same thick orbits.
\end{proof}

In addition to the action of $\Gamma$ on $X$, as in~\cite{braun21} we need $X$ to be equipped with a $\Gamma$-invariant probability measure, $\mu$; \cite{hjorth06} shows that there exists a free, continuous action of $\Gamma$ on $X$ that admits such a $\mu$.  This means that if $\Gamma(A \times S)$ is a thick orbit, and $S$ is $d$-dimensional, we can define (with some abuse of notation)
\[\Area_d(\Gamma(A \times S)) = \Area_d(A \times S) = \mu(A) \cdot \Area_d(S).\]
This allows us to select our $1$-separating filtrations to be (nearly) area-minimizing at each level $\widetilde{Z}_i$, as we did when working directly with $M$.  It also allows us to turn a $1$-separating filtration into a fundamental cycle for $M$, by triangulating the sheets in a consistent way and projecting to $M$, with coefficients given by the measures of the relevant sets in $X$.  We record all of the Cantor bundle setup in the following definition.

\begin{definition}\label{def}
Henceforth, $M$ is a closed, oriented, connected $n$-dimensional Riemannian manifold.  $\Gamma$ is the fundamental group of $M$, and $\widetilde{M}$ is the universal cover of $M$.  We consider a free, continuous action of $\Gamma$ on the Cantor set $X$, with a $\Gamma$-invariant probability measure $\mu$ on $X$.  Then $\Gamma$ acts on $X \times \widetilde{M}$ by the corresponding diagonal action.
\end{definition}

\subsection{Adapting the proof to Cantor bundles}

We define a $\Gamma$-equivariant \textbf{\textit{thick triangulation}} $\widetilde{Q}$ of $X \times \widetilde{M}$ to consist of the following data:
\begin{itemize}
\item A pure $n$-dimensional polyhedron $Q$ (with finitely many faces) with a continuous map $\varphi \co Q \rightarrow M$; and
\item For each face $F$ of $Q$, a $\Gamma$-equivariant choice of clopen subset $A_F \subseteq X$ for each lift of $\varphi\vert_{F}$ to $\widetilde{M}$.
\end{itemize}
It is required to satisfy the following constraints:
\begin{itemize}
\item As $F$ ranges over the (relatively open) faces of $Q$, the $\Gamma$-orbits of the various thick faces $A_F \times \widetilde{\varphi\vert_{F}}(F)$ are non-self-intersecting and form a partition of $X \times \widetilde{M}$;
\item For every $x \in X$, the leaf $x \times \widetilde{M}$ is smoothly triangulated by its intersections with the various thick faces; and
\item If $F$ is a face in the closure of a face $E$, for any lift of $\varphi\vert_{\overline{E}}$ to $\widetilde{M}$, the clopen subsets $A_E$ and $A_F$ corresponding to the restrictions of this lift to $E$ and $F$ satisfy $A_E \subseteq A_F$.
\end{itemize}
It is important that the last inclusion of sets is not required to be an equality; otherwise, thick triangulations would be no more interesting than triangulations of disjoint finite covering spaces of $M$.  A \textbf{\textit{thick subcomplex}} of $\widetilde{Q}$ is the union of orbits of the thick faces of a chosen subcomplex of $Q$.  For a thick subcomplex $\widetilde{P}$ to be \textbf{\textit{pure $d$-dimensional}}, we require not only that the corresponding subcomplex $P$ of $Q$ be pure $d$-dimensional, but also that every point in each thick orbit of $\widetilde{P}$ be in the closure of a $d$-dimensional thick face---that is, of a thick face corresponding to a $d$-dimensional face of $Q$.  We define a $\Gamma$-equivariant \textit{\textbf{thick $d$-polyhedron}} $\widetilde{P}$ in $X \times \widetilde{M}$ to be any pure $d$-dimensional thick subcomplex of a thick triangulation of $X \times \widetilde{M}$.  

A \textbf{\textit{refinement}} of a thick polyhedron $\widetilde{P}$ is a thick polyhedron $\widetilde{P'}$, such that the closure of every orbit of a thick face of $\widetilde{P}$ is a thick subcomplex of $\widetilde{P'}$.  A \textit{\textbf{codimension $1$ thick subpolyhedron}} of $\widetilde{P}$ is a thick $(d-1)$-polyhedron $\widetilde{Z}$ in $X \times \widetilde{M}$ that is a thick subcomplex of some refinement of $\widetilde{P}$, satisfying the \textbf{\textit{smooth nesting condition}}: the relative interior of every thick face of $\widetilde{Z}$ is contained in the relative interior of a thick face of $\widetilde{P}$ of strictly greater dimension.

Given $R > 0$, we say that a codimension $1$ thick subpolyhedron $\widetilde{Z}$ of $\widetilde{P}$ is \textbf{\textit{$R$-separating}} in $\widetilde{P}$ if there is some refinement of $\widetilde{P}$ that contains a refinement of $\widetilde{Z}$ as a thick subcomplex and satisfies the following property.  Abusing notation, we let $P$ denote the polyhedron corresponding to the chosen refinement of $\widetilde{P}$, and let $Z$ denote its subcomplex corresponding to the chosen refinement of $\widetilde{Z}$.  We require that for each connected component of $P \setminus Z$, the restriction of $\varphi \co P \rightarrow M$ to this connected component lifts to $\widetilde{M}$, and its image is contained in a ball of radius $R$ in $\widetilde{M}$.  This condition is almost the same as the formally weaker condition that every connected component of $\widetilde{P} \setminus \widetilde{Z}$ be contained in a ball of radius $R$ in some sheet $\{x\} \times \widetilde{M}$, but the proof of Lemma~\ref{lem:triangulate-cantor} uses the stronger condition.  We define a $\Gamma$-equivariant \textbf{\textit{thick $R$-separating filtration}} of $X \times \widetilde{M}$ to consist of nested thick polyhedra
\[X \times \widetilde{M} = \widetilde{Z}_n \supseteq \widetilde{Z}_{n-1} \supseteq \cdots \supseteq \widetilde{Z}_1 \supseteq \widetilde{Z}_0\]
such that each $\widetilde{Z}_i$ is an $R$-separating codimension $1$ thick subpolyhedron of $\widetilde{Z}_{i+1}$.  The analogue of Lemma~\ref{lem:triangulate} is the following.

\begin{lemma}\label{lem:triangulate-cantor}
Let $M$, $\Gamma$, and $X$ be as in Definition~\ref{def}.  Consider a $\Gamma$-invariant thick $R$-separating filtration
\[X \times \widetilde{M} = \widetilde{Z}_n \supseteq \widetilde{Z}_{n-1} \supseteq \cdots \supseteq \widetilde{Z}_1 \supseteq \widetilde{Z}_0.\]
There is a fundamental cycle $z = \sum_{i} a_i \sigma_i$ for $M$ with coefficients $a_i \in \mathbb{R}$, with a $\pi_1$-killing vertex coloring, such that
\[\sum_{\mathrm{rainbow}\ \sigma_i} \abs{a_i} = 2^n \cdot \Area_0\left(\widetilde{Z}_0\right).\]
\end{lemma}

\begin{proof}
Similar to Lemma~\ref{lem:triangulate}, we want a thick triangulation of $X \times \widetilde{M}$ such that (a refinement of) each $\widetilde{Z}_i$ is a thick subcomplex of the thick triangulation.  We say that a refinement is \textbf{\textit{vertical}} if it subdivides thick faces in the $X$ coordinate only, and \textbf{\textit{horizontal}} if it subdivides thick faces in the $\widetilde{M}$ coordinate only.  Our plan is to take vertical refinements first, until the polyhedra $Z_i$ corresponding to the various $\widetilde{Z}_i$ are nested; then, simultaneously triangulating the nested Riemannian polyhedra
\[Z_n \supseteq Z_{n-1} \supseteq \cdots \supseteq Z_1 \supseteq Z_0\]
as in Lemma~\ref{lem:triangulate}, we obtain our desired thick triangulation as a horizontal refinement of the vertical refinement.

Abusing notation, for each face $F$ of any polyhedron $P \in \{Z_i\}_{i=1}^n$, with corresponding map $\varphi \co P \rightarrow M$, we denote one lift $\widetilde{\varphi\vert_F}(F)$ by $F$, so that $F$ is itself a subset of $\widetilde{M}$.  Then the orbit of the thick face $A_F \times F$ consists of the thick faces $A_{\gamma F} \times \gamma F$ as $\gamma$ ranges over $\Gamma$, and $A_{\gamma F} = \gamma A_F$.  We first claim that there are vertical refinements of all $\widetilde{Z}_i$ such that every pair $\widetilde{Z}_i \subseteq \widetilde{Z}_{i+1}$ satisfies the \textbf{\textit{tight nesting condition}}: not only is every thick face $A_{\gamma F} \times \gamma F$ of $\widetilde{Z}_i$ contained in a thick face $A_{\gamma' E} \times \gamma' E$ of $\widetilde{Z}_{i+1}$ with $\dim E > \dim F$, but we also have $A_{\gamma' E} = A_{\gamma F}$.  To prove this, we prove by induction on $k$ that these simultaneous vertical refinements can be done on the intersections of the various $\widetilde{Z}_i$ with the $k$-skeleton of $\widetilde{Z}_n$.

Let $\widetilde{Q}^{(k)}$ denote the $k$-skeleton of $\widetilde{Z}_n$, and let $Q^{(k)}$ denote the corresponding polyhedron.  Let $E$ be a $k$-dimensional face of $Q^{(k)}$.  By the inductive hypothesis we assume that we have already found vertical refinements of $\widetilde{Q}^{(k-1)} \cap \widetilde{Z}_i$ for all $i$, that satisfy the tight nesting condition.  Our goal is, for each thick face $A_{\gamma F} \times \gamma F$ (possibly equal to $A_E \times E$) of any $\widetilde{Z}_i$ such that $A_{\gamma F} \times \gamma F \subseteq A_E \times E$, to subdivide $A_F$ (and, equivariantly, its translations) in a way that satisfies the tight nesting condition.

Consider $A_E \times \overline{E}$.  Similar to Proposition~\ref{prop:finiteness}, we partition $A_E$ into sets $A_1, \ldots, A_r$, with the informal idea that for any two elements $x, y$ of the same set $A_j$, the leaves $x \times \overline{E}$ and $y \times \overline{E}$ should look identical in terms of the existing stratifications of $\widetilde{Q}^{(k-1)}$ and the various sets $\widetilde{Z}_i$.  More precisely, we let $A_1, \ldots, A_r$ be the minimal nonempty subsets of the algebra generated by the following subsets of $A_E$ under union, intersection, and complement:
\begin{itemize}
\item For each thick face $A_{\gamma F} \times \gamma F$ of $\widetilde{Q}^{(k-1)}$ that intersects the relative boundary $A_E \times \partial E$, we take $A_{\gamma F} \cap A_E$; and
\item For each thick face $A_{\gamma F} \times \gamma F$ of any $\widetilde{Z}_i$ such that $A_{\gamma F} \times \gamma F \subseteq A_E \times E$, we take $A_{\gamma F}$.
\end{itemize}

There are finitely many such thick faces $A_{\gamma F} \times \gamma F$, because each of $\widetilde{Q}^{(k-1)}$ and $\widetilde{Z}_i$ consist of finitely many orbits of thick faces, and only finitely many thick faces of each orbit intersect the bounded set $A_E \times \overline{E}$.  The minimal subsets of the algebra are obtained by taking intersections of generators of the algebra and their complements; thus, the number $r$ of minimal nonempty subsets of the algebra is at most $2$ to the number of generators, which in particular is finite.  Note also that the sets $A_1, \ldots, A_r$ partition $A_E$.

In our vertical refinement of $\widetilde{Q}^{(k)}$, we subdivide $A_E \times E$ into $A_1 \times E, \ldots, A_r \times E$; in $Q^{(k)}$, this replaces $E$ by $r$ identical copies of $E$.  We have chosen the sets $A_1, \ldots, A_r$ to be small enough that if $A_{\gamma F} \times {\gamma F}$ is any face of $\widetilde{Q}^{(k-1)}$ touching the boundary of some $A_j \times E$, then $A_j \subseteq A_{\gamma F}$.  Thus, after doing this for all $k$-faces $E$, the resulting vertical refinement $\widetilde{Q}^{(k)}$ satisfies the conditions needed to be a thick polyhedron.

In our vertical refinement of each $\widetilde{Q}^{(k)} \cap \widetilde{Z}_i$, whenever a thick face $A_{\gamma F} \times \gamma F$ in the orbit of $A_F \times F$ satisfies $A_{\gamma F} \times \gamma F \subseteq A_E \times E$, we subdivide $A_F \times F$ in the following way.  We find the sets $A_j$ that satisfy $A_j \subseteq A_{\gamma F}$, and then translate to get $\gamma^{-1}A_j \subseteq A_F$, so that the various sets $\gamma^{-1}A_j \times F$ subdivide $A_F \times F$.  We have chosen $A_1, \ldots, A_r$ such that $A_{\gamma F}$ is indeed the union of some collection of these sets $A_j$, and so $A_F$ is indeed the union of some collection of $\gamma^{-1}A_1, \ldots, \gamma^{-1}A_r$.  This refinement of each thick face $A_F \times F$ of $\widetilde{Q}^{(k)} \cap \widetilde{Z}_i$ guarantees that the thick polyhedra $\widetilde{Q}^{(k)} \cap \widetilde{Z}_i$ satisfy the tight nesting condition.  This completes the induction, so we may assume that the levels $\widetilde{Z}_i$ satisfy the tight nesting condition.

The smooth nesting condition guarantees that each face of $Z_i$ maps into a face of $Z_{i+1}$; furthermore, the tight nesting condition guarantees that the resulting map embeds $Z_i$ into $Z_{i+1}$ injectively.  Without first doing a vertical refinement, we would still have smooth maps from $Z_i$ into $Z_{i+1}$, but they would be immersions rather than embeddings, because disjoint thick faces of $\widetilde{Z}_i$ could be contained in the same thick face of $\widetilde{Z}_{i+1}$ in such a way that their projections cross.  Thus we have nested polyhedra
\[Z_n \supseteq Z_{n-1} \supseteq \cdots \supseteq Z_1 \supseteq Z_0,\]
where each pair $Z_i$ and $Z_{i+1}$ satisfy the smooth nesting condition.

The proof of Lemma~\ref{lem:triangulate} then guarantees that $Z_n$ can be triangulated such that each $Z_i$ is a subcomplex of the triangulation.  This results in a horizontal refinement of $\widetilde{Z}_n$, in which every $\widetilde{Z}_i$ is a thick subcomplex.

We construct our fundamental cycle $z$ as the barycentric subdivision of the cycle $z'$, constructed as follows.  For each $n$-dimensional simplex $F$ of $Z_n$, we have a corresponding singular simplex $\varphi\vert_F$ in $M$, which is immersed.  Using a fixed orientation of $M$, and an arbitrary total ordering of the vertices of $Z_n$, the immersion is either orientation-preserving or orientation-reversing on $F$.  We let the coefficient of the singular simplex $\varphi\vert_F$ in $z'$ be $\mu(A_F)$ if the immersion is orientation-preversing, or $-\mu(A_F)$ if it is orientation-reversing.

To obtain a $\pi_1$-killing vertex coloring of the barycentric subdivision $z$ of $z'$, we color $Z_n$ by dividing it first into levels $Z_i \setminus Z_{i-1}$ for $i = 0, \ldots, n$, and then dividing each level into connected components, and each of these connected components corresponds to one color.  In this way, the relative interior of each simplex of $Z_n$ is only one color.  Thus, each vertex of $z$, as the barycenter of some simplex of $Z_n$, is colored by the color of that simplex.  The $R$-separating property guarantees that any loop within a single color has null-homotopic image in $M$, which means that the coloring is $\pi_1$-killing.

For each thick point of $\widetilde{Z}_0$, the total weight of the rainbow simplices touching that point is $2^n$ times the weight of that point; this can be seen by using Proposition~\ref{prop:finiteness} to divide a neighborhood of the thick point into layers, each of which looks like the situation of Lemma~\ref{lem:triangulate}.  Thus, taking the sum over all rainbow simplices gives the desired inequality.
\end{proof}

We prove analogues of Lemmas~\ref{lem:replacement} and~\ref{lem:volball} except that instead of applying the comparisons to a single ball of changing radius, we apply them to a disjoint union of thick balls, all of the same changing radius.  To describe this, we give the following definitions.  For any $r \in \mathbb{N}$, given clopen sets $A_1, \ldots, A_r \subseteq X$ and points $\widetilde{p}_1, \ldots, \widetilde{p}_r \in \widetilde{M}$, we refer to the resulting finite collection of thick points $A_1 \times \widetilde{p}_1, \ldots, A_r \times \widetilde{p}_r$ as a \textbf{\textit{multi-point}} $D$.  For each $\rho > 0$, we use the notation $D(\rho)$ to denote the finite collection of thick balls $A_1 \times B(\widetilde{p}_1, \rho), \ldots, A_r \times B(\widetilde{p}_r, \rho)$, and we refer to $D(\rho)$ as a \textbf{\textit{multi-ball}}.  The analogue of Lemma~\ref{lem:replacement} is the following.

\begin{lemma}\label{lem:replacement-cantor}
Let $M$, $\Gamma$, and $X$ be as in Definition~\ref{def}, and let $R > 0$ be arbitrary.  Let $\widetilde{P}$ be a pure $d$-dimensional $\Gamma$-equivariant thick polyhedron in $X \times \widetilde{M}$, and let $\widetilde{Z}$ be an $R$-separating subpolyhedron of $\widetilde{P}$.  Suppose that $\Area_{d-1}(\widetilde{Z})$ is within $\varepsilon$ of the infimal area of $R$-separating subpolyhedra of $\widetilde{P}$.  Let $D(R)$ be a multi-ball, such that the thick balls that make up $D(R)$ have disjoint and non-self-intersecting orbits.  Then for all $r_1, r_2$ with $0 < r_1 < r_2 < R$, we have
\[\int_{r_1}^{r_2} \Area_{d-1}  \left(\widetilde{Z} \cap D(\rho)\right) d\rho \leq  \Area_d  \left(\widetilde{P}  \cap D(r_2) \setminus D(r_1)\right)+ 2\varepsilon R.\]
\end{lemma}

\begin{proof}
The proof is the same as that of Lemma~\ref{lem:replacement}, but instead of considering replacing $Z$ inside a single approximate ball $\widehat{B}(p, \rho)$, we consider $\Gamma$-equivariantly replacing $\widetilde{Z}$ inside a disjoint union of finitely many non-self-intersecting orbits of thick approximate balls $A_i \times \widehat{B}(\widetilde{p}_i, \rho)$.  The function $\Area_d(D(\rho))$ is differentiable for almost every $\rho$; we claim that for almost every $\rho$, the thick subpolyhedron $\widetilde{Z'}$ of $\widetilde{P}$ obtained by replacing each translate of $\widetilde{Z} \cap \left(A_i \times \widehat{B}(\widetilde{p}_i, \rho)\right)$ by the corresponding translate of $\widetilde{P} \cap \left(A_i \times \widehat{S}(\widetilde{p}_i, \rho)\right)$ is an $R$-separating thick subpolyhedron of $\widetilde{P}$.

To prove this, consider the refinement of $\widetilde{P}$ that has $\widetilde{Z}$ as a subcomplex; we refer to this refinement simply as $\widetilde{P}$.  Similar to the proof of Lemma~\ref{lem:triangulate-cantor} we start by vertically refining $\widetilde{P}$ in the following way.  For each face $F$ of the corresponding polyhedron $P$, we find all pairs of $i \in \{1, \ldots, r\}$ and $\gamma \in \Gamma$ such that $\gamma B(\widetilde{p}_i, R)$ intersects $F$; there are finitely many such pairs.  We partition the thick face $A_F \times F$ by partitioning $A_F$ into the minimal sets of the algebra generated under union, intersection, and complement by the various sets $A_F \cap \gamma A_i$. Each minimal set is the intersection of various $A_F \cap \gamma A_i$ and/or $A_F \setminus \gamma A_i$, so there are finitely many minimal sets.  Doing this process for every face of $P$ gives a vertical refinement; again abusing notation, we refer to the refinement and its corresponding polyhedron simply as $\widetilde{P}$ and $P$.  For each face of $P$, various approximate-distance-to-$\widetilde{p}_i$ functions may be defined on disjoint portions of that face.  That is, if we consider the copy of original face $F$ corresponding to a particular set of $(i, \gamma)$ pairs, for each pair $(i, \gamma)$, the approximate distance to $\widetilde{p}_i$ is defined on $\gamma^{-1}F \cap B(\widetilde{p}_i, R)$, and by translating we may define it on $F \cap \gamma B(\widetilde{p}_i, R)$.  Patching the faces together, we obtain a smooth function defined on an open subset of $P$.  We choose $\rho$ to be a regular value of this smooth function.

Let $S_\rho \subseteq P$ denote the level set at value $\rho$ of this function on $P$.  Then $P$ is a Whitney stratified space, where the strata are the intersections of the faces of $P$ with $S_\rho$ or its complement, so the stratification can be refined to give a smooth triangulation~\cite{goresky88, goresky78}.  The resulting horizontal refinement of $\widetilde{P}$ has $\widetilde{Z'}$ as a thick subcomplex.  Furthermore, $\widetilde{Z'}$ is $R$-separating in $\widetilde{P}$, because for every connected component of $P \setminus Z'$, the corresponding lifted image in $\widetilde{M}$ either is contained in some $B(\widetilde{p}_i, R)$ or is contained in the lifted image in $\widetilde{M}$ of a connected component of the old $P \setminus Z$.  The remainder of the proof is the same as that of Lemma~\ref{lem:replacement}.
\end{proof}

The analogue of Lemma~\ref{lem:volball} is the following.

\begin{lemma}\label{lem:volball-cantor}
Let $M$, $\Gamma$, and $X$ be as in Definition~\ref{def}.  For all $R > 0$ and all $\varepsilon > 0$, there exists a $\Gamma$-equivariant thick $R$-separating filtration
\[X \times \widetilde{M} = \widetilde{Z}_n \supseteq \widetilde{Z}_{n-1} \supseteq \cdots \supseteq \widetilde{Z}_1 \supseteq \widetilde{Z}_0,\]
such that the following is true.  Let $D(R)$ be a multi-ball, such that the thick balls that make up $D(R)$ have disjoint and non-self-intersecting orbits.  Then for all $r_1, r_2$ with $0 < r_1 < r_2 < R$, we have
\[\Area_0\left(\widetilde{Z}_0 \cap D(r_1)\right) \cdot \frac{(r_2 - r_1)^n}{n!} \leq \Area_n(D(r_2)) + \varepsilon.\]
\end{lemma}

\begin{proof}
The proof follows from Lemma~\ref{lem:replacement-cantor} in the same way that the proof of Lemma~\ref{lem:volball} follows from Lemma~\ref{lem:replacement}.
\end{proof}

We divide the remainder of the proof of the main theorem into Lemmas~\ref{lem:separate-thick} and~\ref{lem:inequality-string}, as follows.

\begin{lemma}\label{lem:separate-thick}
Let $M$, $\Gamma$, and $X$ be as in Definition~\ref{def}.  There exists $m \in \mathbb{N}$ such that the following is true.  Let $\widetilde{Z} \subseteq X \times \widetilde{M}$ be a $0$-dimensional $\Gamma$-equivariant thick polyhedron.  Then we can find a multi-point $D$ with the following properties:
\begin{enumerate}
\item $ D \subseteq \widetilde{Z} \subseteq \Gamma\left(D\left(\frac{1}{2}\right)\right)$.
\item The orbits of thick balls of $ D\left(\frac{1}{4}\right)$ are disjoint and non-self-intersecting.
\item $D$ can be partitioned into multi-points $D_1, \ldots, D_m$ such that for each $\ell = 1, \ldots, m$, the orbits of thick balls of the multi-ball $\displaystyle D_{\ell}(1)$ are disjoint and non-self-intersecting.
\end{enumerate}
\end{lemma}

\begin{proof}
First we claim that $X$ can be partitioned into finitely many clopen sets $Y_1, \ldots, Y_s$, with the property that for every $j = 1, \ldots, s$ and every $\widetilde{p} \in \widetilde{M}$, the thick ball $Y_j \times B(\widetilde{p}, 1)$ has non-self-intersecting orbit.  To show this, we observe that there are finitely many elements of $\Gamma$ that move any points of $\widetilde{M}$ a distance of at most $2$.  As in~\cite{braun21}, if we put a metric on $X$, because the action of $\Gamma$ on $X$ is free and $X$ is compact, there is some minimum distance that the points in $X$ are moved by these finitely many elements.  Thus, we can partition $X$ into clopen sets $Y_1, \ldots, Y_s$ that each have diameter less than this minimum distance, and these sets have the desired property.

Next, let $A_1\times \widetilde{p}_1, \ldots, A_r \times \widetilde{p}_r$ be thick points such that $\widetilde{Z}$ is the union of their orbits.  We use the following process to construct $D$.  For each $i = 1, \ldots, r$ and each $j = 1, \ldots, s$, we recursively define the clopen set $E_{ij} \subseteq X$ to be the set of all $x \in A_i \cap Y_j$ for which the leaf $x \times B(\widetilde{p}_i, \frac{1}{4})$ does not intersect any translate of any earlier $E_{i'j'} \times B(\widetilde{p}_{i'}, \frac{1}{4})$, which are those with $i' < i$ and $j'$ arbitrary, or with $i' = i$ and $j' < i$.

We let $D$ consist of the various $E_{ij} \times \widetilde{p}_i$.  The orbits of the thick balls of $D(\frac{1}{4})$ are non-self-intersecting, because each is contained in some $Y_j \times B(\widetilde{p}_i, 1)$; and, they are disjoint by construction.  We also observe that $D$ is maximal in the sense that for all $x \times \widetilde{p}_i \in \widetilde{Z}$, the leaf $x \times B(\widetilde{p}_i, \frac{1}{4})$ must intersect the orbit of a thick ball of $D(\frac{1}{4})$.  This implies that the orbits of the thick balls of $D(\frac{1}{2})$ together contain all of $\widetilde{Z}$.  Thus, the first two properties of the lemma statement hold.  At this point, we renumber the thick points of $\widetilde{Z}$ and $D$, so that $\widetilde{Z}$ is still the union of the orbits of disjoint thick points $A_1 \times \widetilde{p}_1, \ldots, A_r \times \widetilde{p}_r$ but some $\widetilde{p}_i$ may be equal to each other, and $D$ is the union of disjoint thick points $E_1\times \widetilde{p}_1, \ldots, E_r \times \widetilde{p}_r$, with the following properties:
\begin{itemize}
\item for each $i$ we have $E_i \subseteq A_i$,
\item  the orbits of thick balls $E_i \times B\left(\widetilde{p}_i, \frac{1}{4}\right)$ are disjoint and non-self-intersecting, and 
\item the orbits of thick balls $E_i \times B\left(\widetilde{p}_i, 1\right)$ are non-self-intersecting but not necessarily disjoint.
\end{itemize}

To partition $D$ to satisfy the third property of the lemma statement, we use a greedy coloring argument.  Informally, for each leaf $x \times B(\widetilde{p}_i, 1)$ of $D(1)$, we assign that leaf to the set $D_{\ell}$ with the least possible $\ell$ such that this leaf does not intersect the orbit of any other leaf already allocated to $D_\ell$.  To formalize this idea, for each $i = 1, \ldots, r$ and each $j \in \mathbb{N}$ in order, we recursively define the clopen set $F_{i, j}$ to be the set of all $x \in E_i$ with the property that for all $j' < j$ there exists $i' < i$ such that $x \times B(\widetilde{p}_i, 1)$ intersects the orbit of $(F_{i', j'} \setminus F_{i', j'+1}) \times B(\widetilde{p}_{i'}, 1)$.  (Informally, this is the part of $E_i$ for which the corresponding $\ell$ has to be at least $j$.)  Then, we define $D_{\ell}$ such that $D_{\ell} \cap (E_{i} \times \widetilde{p}_i)$ is $(F_{i, \ell} \setminus F_{i, \ell+1}) \times \widetilde{p}_i$.  By construction, the orbits of the thick balls in each $D_{\ell}(1)$ are disjoint.

We claim that there exists $m$, depending only on $M$ and not on $\widetilde{Z}$, such that for all $\ell > m$ the resulting set $D_{\ell}$ is empty.  Let $m$ be the maximum number such that in $\widetilde{M}$ there exists a ball of radius $\frac{9}{4}$ that contains $m$ disjoint balls of radius $\frac{1}{4}$.  This number $m$ is finite because $M$ is compact; for instance, $m$ is no more than the number of balls of radius $\frac{1}{4}$ needed to cover the $\frac{9}{4}$-neighborhood of a fixed fundamental domain for $M$ in $\widetilde{M}$, because the center of the $\frac{9}{4}$-ball can be taken to be inside the fundamental domain, and the centers of the disjoint $\frac{1}{4}$-balls are contained in distinct $\frac{1}{4}$-balls of the open cover.  

Suppose for the sake of contradiction that some $x \in E_i$ is in $F_{i, m+1}$.  This means that there are at least $m$ translates $\gamma_{i'}(E_{i'} \times B(\widetilde{p}_{i'}, 1))$ that intersect $x \times B(\widetilde{p}_i, 1)$; in particular, their projections to the $X$ direction contain $x$.  However, in earlier steps we have constructed the corresponding thick balls of radius $\frac{1}{4}$ to be disjoint.  Thus, the balls $B(\gamma_{i'}\widetilde{p}_{i'}, \frac{1}{4})$, together with $B(\widetilde{p}_i, \frac{1}{4})$, are a family of at least $m+1$ disjoint balls inside the ball $B(\widetilde{p}_i, \frac{9}{4})$, giving a contradiction.  We conclude that there are no more than $m$ sets $D_{\ell}$ in our partition of $D$.
\end{proof}

The following lemma estimates $\Area_0\left(\widetilde{Z}_0\right)$ by adding up the contributions from the multi-balls of Lemma~\ref{lem:separate-thick}, similar to the proof of Theorem~\ref{thm:baby}.

\begin{lemma}\label{lem:inequality-string}
Let $M$, $\Gamma$, and $X$ be as in Definition~\ref{def}.  Let $m$ be as in Lemma~\ref{lem:separate-thick}.  For $R = 1$ and any $\varepsilon > 0$, consider the $\Gamma$-equivariant thick $R$-separating filtration guaranteed by Lemma~\ref{lem:volball-cantor}.  Then the $0$-dimensional level of this filtration satisfies
\[\Area_0\left(\widetilde{Z}_0\right) \leq 2^n \cdot n! \cdot \left(m\varepsilon + 4^n\cdot n! \cdot V_1(\widetilde{M}) \cdot (\Vol M + \varepsilon)\right).\]
\end{lemma}

\begin{proof}
We use the multi-points $D_1, \ldots, D_m$ from applying Lemma~\ref{lem:separate-thick} to $\widetilde{Z}_0$.  Taking $r_1 = \frac{1}{2}$ and $r_2 \rightarrow 1$ gives
\[\Area_0\left(\widetilde{Z}_0 \cap D_\ell\left(\frac{1}{2}\right)\right) \cdot \frac{(1/2)^n}{n!} \leq \Area_n(D_\ell(1)) + \varepsilon\]
for each $\ell = 1, \ldots, m$, and we also know
\[\Area_n(D_\ell(1)) \leq V_1(\widetilde{M}) \cdot \Area_0(D_\ell).\]
Then, taking $r_1 \rightarrow 0$ and $r_2 = \frac{1}{4}$ with $D = \bigcup_{i = 1}^m D_i$ gives
\[\Area_0 \left(\widetilde{Z}_0 \cap D\right) \cdot \frac{(1/4)^n}{n!} \leq \Area_n \left(D\left(\frac{1}{4}\right)\right) + \varepsilon.\]
The properties from Lemma~\ref{lem:separate-thick} imply
\[\sum_{i=\ell}^m \Area_0(D_\ell) = \Area_0\left(\widetilde{Z}_0 \cap D\right),\]
\[ \Area_0\left(\widetilde{Z}_0\right) \leq \sum_{\ell =1}^m \Area_0\left(\widetilde{Z}_0 \cap D_\ell\left(\frac{1}{2}\right)\right),\]
\[\Area_n\left(D\left(\frac{1}{4}\right)\right) \leq \Vol M,\]
and for each $\ell = 1, \ldots, m$,
\[\Area_n(D_\ell(1)) \leq V_1(\widetilde{M}) \cdot \Area_0(D_\ell).\]
Putting the inequalities together gives
\begin{align*}
\Area_0\left(\widetilde{Z}_0\right) & \leq  \sum_{\ell = 1}^m \Area_0\left(\widetilde{Z}_0 \cap D_\ell\left(\frac{1}{2}\right)\right)\leq \\
& \leq  2^n\cdot n! \cdot \sum_{\ell = 1}^m\left( \Area_n(D_\ell(1)) + \varepsilon\right) \leq \\
& \leq  2^n \cdot n! \cdot \left(m\varepsilon + V_1(\widetilde{M}) \cdot \sum_{\ell = 1}^m \Area_0(D_\ell)\right) =\\
& =  2^n \cdot n! \cdot \left(m\varepsilon + V_1(\widetilde{M}) \cdot \Area_0\left(\widetilde{Z}_0 \cap D\right)\right)\leq \\
& \leq  2^n \cdot n! \cdot \left(m\varepsilon + 4^n \cdot n! \cdot V_1(\widetilde{M}) \cdot \left(\Area_n\left(D\left(\frac{1}{4}\right)\right) + \varepsilon\right)\right) \leq\\
& \leq 2^n \cdot n! \cdot \left(m\varepsilon + 4^n\cdot n! \cdot V_1(\widetilde{M}) \cdot (\Vol M + \varepsilon)\right).
\end{align*}
\end{proof}

Combining Lemma~\ref{lem:inequality-string} with Lemma~\ref{lem:triangulate-cantor} completes the proof of our main theorem, as follows.

\begin{proof}[Proof of Theorem~\ref{thm:big}]
To prove the statement about the case $V_1(\widetilde{M}) < \frac{1}{n!}$, we apply Lemma~\ref{lem:volball-cantor} with $R = 1$ and $\varepsilon < \frac{1}{n!} - V_1(\widetilde{M})$, and then apply Lemma~\ref{lem:separate-thick} to the resulting set $\widetilde{Z}_0$  to find $D_1, \ldots, D_m$.  If any of these $D_\ell$ has nonzero measure, the conclusion of Lemma~\ref{lem:volball-cantor} with $r_1 \rightarrow 0$ and $r_2 \rightarrow 1$ gives a contradiction.  Thus, $\widetilde{Z}_0$ is empty, and by Lemma~\ref{lem:triangulate-cantor} and Corollary~\ref{thm:rainbow-cor} the image of the fundamental class of $M$ is zero in $H_*(B\Gamma; \mathbb{Q})$.

In the case where $V_1(\widetilde{M}) \geq \frac{1}{n!}$, for every $\varepsilon > 0$ we apply Lemma~\ref{lem:volball-cantor} with $R = 1$, and then apply Lemma~\ref{lem:triangulate-cantor} to find a fundamental cycle for $M$ with total contribution from rainbow simplices equal to $2^n \cdot \Area_0\left(\widetilde{Z}_0\right)$.  Substituting Lemma~\ref{lem:inequality-string} implies
\[\Vert M \Vert_{\Delta} \leq 2^n \cdot \Area_0\left(\widetilde{Z}_0\right) \leq 4^n\cdot n!\left(m\varepsilon + 4^n\cdot n!\cdot V_1(\widetilde{M})\cdot \left(\Vol M + \varepsilon\right)\right),\]
so taking $\varepsilon \rightarrow 0$ gives
\[\Vert M \Vert_{\Delta} \leq 16^n(n!)^2 \cdot V_1(\widetilde{M}) \cdot \Vol M.\]
\end{proof}

\section{Integral foliated simplicial volume}\label{sec:foliated}

We have succeeded in proving Theorem~\ref{thm:big}, which corresponds to Theorems~1.1 and~1.3 of~\cite{braun21} and is the main goal of that paper.  In the course of their proof, Braun and Sauer deduce some additional corollaries, which are Theorems~1.4 and~1.5 of their paper.  In this section we explore whether our new proof method also produces these corollaries.  The answer is that it does, except that we need a small hypothesis not included in~\cite{braun21}: we assume that the fundamental group $\Gamma$ does not have any torsion.

These corollaries depend on using coefficients in $\mathbb{Z}$ rather than in $\mathbb{R}$.  Our proof above uses $\mathbb{R}$-coefficients in two ways: once to weight each thickened simplex by the measure of its Cantor set component, and once in the simplex straightening step, where each straightened simplex is chosen to be an average of permuted and translated simplices.  For the first use of $\mathbb{R}$-coefficients, we are in the same situation as Braun and Sauer.  Once we express our proof slightly differently, so that the norm of the fundamental cycle is estimated in the Cantor bundle rather than downstairs in $M$ or $B\Gamma$, we can reproduce their proofs of the corollaries.  For the second use of $\mathbb{R}$-coefficients, we need to eliminate the averaging from the simplex straightening step, which is not a step in the Braun--Sauer proof.  To eliminate the averaging over permutations, we use the work of Campagnolo and Sauer in~\cite{campagnolo19}.  I do not know how to make a similar modification to eliminate the averaging over translations, which is why we introduce the additional hypothesis not assumed by Braun and Sauer.

The corollaries follow from one main estimate, which we state as Theorem~\ref{thm:cor-source} below.  The statements of the theorem and corollaries need several definitions from~\cite{braun21}.  First is the definition of what can be thought of informally as the singular homology of a Cantor bundle.  Let $P$ be a space with fundamental group $\Gamma$, so that $\Gamma$ acts on $X \times \widetilde{P}$.  As usual, $C_d(\widetilde{P}; \mathbb{Z})$ denotes the singular $d$-chains on $\widetilde{P}$ for each degree $d$.  We would like to be able to thicken these by clopen subsets of the Cantor set $X$.  Formally, let $C(X; \mathbb{Z})$ be the set of functions from $X$ to $\mathbb{Z}$ that are continuous---that is, the preimage of every integer is a clopen set.  Then the homology we are looking for is that of the chain complex
\[C(X; \mathbb{Z}) \otimes_{\mathbb{Z}[\Gamma]} C_d(\widetilde{P}; \mathbb{Z}).\]
Given a cycle $z = \sum_{i} f_i \otimes \sigma_i$ in this chain complex, where none of the $\sigma_i$ are $\Gamma$-translates of each other, the analogue of the $\ell^1$ norm of $z$ is
\[\abs{z}_1 = \sum_i \int \abs{f_i}\ d\mu.\]
Then the \textbf{\textit{$X$-parametrized integral $\ell^1$-semi-norm}} of a homology class of this chain complex is the infimum of $\abs{z}_1$ over all cycles $z$ representing the homology class.  The $X$-parametrized integral $\ell^1$-semi-norm of the class $[z]$ is denoted by $\Vert [z] \Vert_{\mathbb{Z}}^X$.  Given any singular cycle on $P$ with $\mathbb{Z}$-coefficients, we can build a corresponding cycle in the chain complex we have just constructed, by replacing each simplex $\sigma_i$ by $1 \otimes \sigma_i$.  We denote this map on homology by $j_*^{P}$.

The theorem below gives an upper bound for the quantity $\Vert j_*^{B\Gamma}c_*[M]\Vert_{\mathbb{Z}}^X$, where $c\co M \rightarrow B\Gamma$ is the classifying map.  This means, then, that we take a fundamental cycle of $M$, map it to the classifying space $B\Gamma$, find the corresponding cycle on $X \times E\Gamma$, and construct a homologous cycle (with integer coefficients) for which the norm is sufficiently small.  As before, we let $V_1(\widetilde{M})$ denote the supremal volume of balls of radius $1$ in the universal cover of $M$.

\begin{theorem}\label{thm:cor-source}
Let $M$ be a closed, oriented, connected $n$-dimensional Riemannian manifold.  Let $\Gamma = \pi_1(M)$, suppose that $\Gamma$ is torsion-free, and consider a free, continuous action of $\Gamma$ on the Cantor set $X$, with a $\Gamma$-invariant probability measure.  Then the $X$-parametrized integral $\ell^1$-semi-norm of the image of the fundamental homology class of $M$ under the classifying map satisfies the inequality
\[\Vert j_*^{B\Gamma}c_*[M]\Vert_{\mathbb{Z}}^X \leq (n+1)16^n(n!)^3\cdot V_1(\widetilde{M}) \cdot \Vol M.\]
\end{theorem}

First we state the corollaries and summarize how Braun and Sauer deduce them from this estimate.  Then, at the end of this section we show how to modify our main proof to prove this estimate.

For the first corollary we need the following definitions given in~\cite{braun21} and attributed to~\cite{frigerio16, gromov99, schmidt05}.  Our action of $\Gamma$ on $X$ is an example of a free probability-measure-preserving action $\alpha$ of $\Gamma$ on a \textbf{\textit{standard probability space}}, meaning a Polish space where the set of measurable sets is equal to the Borel $\sigma$-algebra.  For such an action, the norm that we have denoted by $\Vert j_*^{M} [M] \Vert_{\mathbb{Z}}^X$ can also be denoted by $\abs{M}^\alpha$, and the infimum of all $\abs{M}^\alpha$ over such actions of $\Gamma = \pi_1(M)$ on standard probability spaces is called the \textbf{\textit{integral foliated simplicial volume}} of $M$.  The definition in~\cite{frigerio16} uses the chain complex $L^\infty(X; \mathbb{Z}) \otimes_{\mathbb{Z}[\Gamma]} C_d(\widetilde{P}; \mathbb{Z})$ where we have used $C(X; \mathbb{Z}) \otimes_{\mathbb{Z}[\Gamma]} C_d(\widetilde{P}; \mathbb{Z})$, but our upper bound in Theorem~\ref{thm:cor-source} applies to either version of the definition, because the relevant functions on $X$ are constructed to have image in $\{-1, 0, 1\}$ as well as being continuous.  Note that $\Vert M \Vert_{\Delta} \leq \abs{M}^\alpha$ for all $\alpha$, because every cycle used to define $\abs{M}^\alpha$ gives rise to a fundamental cycle of $M$ with real coefficients and the same $\ell^1$-semi-norm, by integrating in the $X$ direction as in~\cite[Remark 3.4]{braun21}, which corresponds geometrically to projecting from $X \times_{\Gamma} \widetilde{M}$ to $M$.

The following corollary is the same as Theorem~1.5 of~\cite{braun21}, because the assumption that $M$ is aspherical implies our hypothesis about $\Gamma$ being torsion-free.

\begin{corollary}[Theorem~1.5 of~\cite{braun21}]\label{cor:first}
Let $M$ be a closed, oriented, connected $n$-dimensional Riemannian manifold, and suppose that $\widetilde{M}$ is contractible.  Then for any free probability-measure-preserving action $\alpha$ of $\Gamma = \pi_1(M)$ on a standard probability space, we have
\[\abs{M}^\alpha \leq (n+1)16^n(n!)^3\cdot V_1(\widetilde{M}) \cdot \Vol M,\]
so $(n+1)16^n(n!)^3\cdot V_1(\widetilde{M}) \cdot \Vol M$ is an upper bound on the integral foliated simplicial volume of $M$.
\end{corollary}

\begin{proof}
Because $M$ is aspherical, we can deduce that $\Gamma$ is torsion-free; see~\cite[Corollary VII.2.5]{brown82}.  Thus we may apply Theorem~\ref{thm:cor-source}.  

The remainder of the proof is identical to the proof in~\cite{braun21}.  Because $\widetilde{M}$ is contractible, we can take $B\Gamma$ to be equal to $M$, so $j_*^{B\Gamma}c_*[M] = j_*^M[M]$.  This proves the desired bound on $\abs{M}^\alpha$ in the case where $\alpha$ is our chosen action of $\Gamma$ on the Cantor set $X$.  Then a theorem from~\cite{elek21} shows that given any other $\alpha$ on any other standard probability space, there is a suitable embedding of $X$ into that space that allows us to deduce the conclusion.
\end{proof}

The second corollary is the same as Theorem~1.4 of~\cite{braun21}, except that in part (1) we add the hypothesis that $\Gamma$ is torsion-free.  For information about the von Neumann rank we refer to~\cite[Definition 3.8]{braun21}, and for information about $\ell^2$-Betti numbers, Braun and Sauer cite L\"uck~\cite{lueck98, lueck02}.

\begin{corollary}[Theorem~1.4 of~\cite{braun21}]\label{cor:second}
Let $M$ be a closed, oriented, connected $n$-dimensional Riemannian manifold.
\begin{enumerate}[(1)]
\item Let $\Gamma = \pi_1(M)$, and suppose that $\Gamma$ is torsion-free.  Then the von Neumann rank of the image $c_*[M] \in H_n(B\Gamma)$ of the fundamental homology class of $M$ under the classifying map is bounded above by $n(n+1)16^n(n!)^3\cdot V_1(\widetilde{M}) \cdot \Vol M$.
\item Suppose that $\widetilde{M}$ is contractible.  Then the $\ell^2$-Betti numbers of $M$ satisfy
\[\beta_i^{(2)}(M) \leq (n+1)16^n(n!)^3\cdot V_1(\widetilde{M}) \cdot \Vol M\]
for every $i \in \mathbb{N}$, and the Euler characteristic satisfies
\[\abs{\chi(M)} \leq n(n+1)16^n(n!)^3\cdot V_1(\widetilde{M}) \cdot \Vol M.\]
\end{enumerate}
\end{corollary}

\begin{proof}
The proof is the same as in~\cite{braun21}.  For the statement about von Neumann rank, Theorem~3.10 of~\cite{braun21} states that the von Neumann rank of a class in degree $d$ is bounded above by $d$ times the corresponding $X$-parametrized integral $\ell^1$-semi-norm.  Our fundamental class is in degree $n$, so the bound on its von Neumann rank is $n$ times the bound from Theorem~\ref{thm:cor-source}.

For the statement about $\ell^2$-Betti numbers, Theorem~3.7 of~\cite{braun21}, attributed to~\cite{schmidt05}, states that every $\ell^2$-Betti number of a closed, oriented, connected manifold is bounded above by the $X$-parametrized $\ell^1$-semi-norm of its fundamental class.  Using the assumption that $M = B\Gamma$, this is the same as the quantity bounded by Theorem~\ref{thm:cor-source}.  An aspherical manifold of dimension $n$ can have nonzero $\ell^2$-Betti numbers only in degrees $1$ through $n$, and the alternating sum of $\ell^2$-Betti numbers equals the Euler characteristic, so summing the bound on all $\ell^2$-Betti numbers yields the bound on Euler characteristic.
\end{proof}

We turn our attention to proving Theorem~\ref{thm:cor-source}, which is the estimate that implies these corollaries.  The proof of Theorem~\ref{thm:big} needs to be modified only slightly to produce this estimate.  The difference is that we need to prove the following Lemma~\ref{lem:thick-rainbow}, which corresponds to Theorem~\ref{thm:rainbow}, so that we can estimate the norm of our thickened triangulation in the Cantor bundle without projecting it down to $M$.  

To state the lemma, we need to adapt the definition of $\pi_1$-killing vertex coloring.  Let $z = \sum_i f_i \otimes \sigma_i$ be a cycle in the chain complex $C(X; \mathbb{Z}) \otimes_{\mathbb{Z}[\Gamma]} C_d(\widetilde{P}; \mathbb{Z})$.  Informally, one way to obtain such a ``singular'' cycle is as the continuous image of a ``simplicial'' cycle.  More precisely, let $Y$ be a polyhedron, and let $\widehat{Y} \rightarrow Y$ be a covering space with deck transformation group $\Gamma$.  Let $C_d^{\Delta}(\widehat{Y}; \mathbb{Z})$ denote the simplicial chain group of the (infinite) simplicial complex $\widehat{Y}$.  Then $C(X; \mathbb{Z}) \otimes_{\mathbb{Z}[\Gamma]} C_d^{\Delta}(\widehat{Y}; \mathbb{Z})$ is a chain complex, and every $\Gamma$-equivariant continuous map $\widehat{\sigma} \co \widehat{Y} \rightarrow \widetilde{P}$ induces a map of chain complexes to $C(X; \mathbb{Z}) \otimes_{\mathbb{Z}[\Gamma]} C_d(\widetilde{P}; \mathbb{Z})$.  Any such map $\widehat{\sigma}$ also determines a continuous map $\sigma \co Y \rightarrow P$.  If $z$ is obtained in this way as the image of a cycle in $C(X; \mathbb{Z}) \otimes_{\mathbb{Z}[\Gamma]} C_d^{\Delta}(\widehat{Y}; \mathbb{Z})$---that is, if each $\sigma_i$ equals $\widehat{\sigma}\vert_{\Delta_i}$ for some simplex $\Delta_i$ of $\widehat{Y}$ and $\sum_{i} f_i \otimes \Delta_i$ forms a cycle---then we say that $Y$ is a \textbf{\textit{domain complex}} for $z$.  Note that every cycle $z$ admits a domain complex, by taking one simplex for each $\sigma_i$ and identifying the $(d-1)$-dimensional faces of two such simplices whenever the corresponding elements of $C_{d-1}(\widetilde{P}; \mathbb{Z})$ differ by a factor in $\mathbb{Q}[\Gamma]$.

Then a $\pi_1$-killing vertex coloring of $z$ is defined to be a coloring of the vertices of any domain complex $Y$ for $z$, such that if we take all the edges of $Y$ for which both vertices are a given color, the restriction of $\sigma$ to this $1$-complex is a null-homotopic map to $P$.  Note that homotoping the map from $Y$ to $P$ results in a cycle homologous to $z$, by homotoping the various $\sigma_i$ while preserving the various $f_i$.  Under such a homotopy, the vertex coloring of $Y$ remains $\pi_1$-killing.


\begin{lemma}\label{lem:thick-rainbow}
Let $P$ be a path-connected CW complex, let $\Gamma = \pi_1(P)$, and suppose that $\Gamma$ is torsion-free.  Consider a free, continuous action of $\Gamma$ on the Cantor set $X$, with a $\Gamma$-invariant probability measure $\mu$.  Abusing notation, let $c \co X \times \widetilde{P} \rightarrow X \times E\Gamma$ be the $\Gamma$-equivariant map corresponding to the classifying map $c \co P \rightarrow B\Gamma$.  Let $z = \sum_i f_i \otimes \sigma_i$ be a cycle in the chain complex $C(X; \mathbb{Z}) \otimes_{\mathbb{Z}[\Gamma]} C_d(\widetilde{P}; \mathbb{Z})$, with a $\pi_1$-killing vertex coloring.  Then the $X$-parametrized integral $\ell^1$-semi-norm of the image of the homology class of $z$ satisfies
\[\Vert c_*[z] \Vert_{\mathbb{Z}}^X \leq (d+1)!\cdot\sum_{\mathrm{rainbow}\ \sigma_i} \int \abs{f_i}\ d\mu.\]
\end{lemma}

\begin{proof}
The proof follows the proof of Theorem~\ref{thm:rainbow}: homotope $c(z)$ so that all vertices of each color and all edges among them go to a single point---or more precisely, a single $\Gamma$-orbit in $X \times E\Gamma$.  Then we straighten, needing to define the notion of straightening in such a way that each non-rainbow simplex maps to $0$ and each rainbow simplex maps to a sum of $(d+1)!$ simplices.  For this notion of straightening, we cite Theorem~3.3 of~\cite{campagnolo19}, which constructs a chain homotopy equivalence given by replacing each $d$-simplex by a signed sum of the $(d+1)!$ simplices of its barycentric subdivision, in such a way that permuting the vertices of the simplex results in multiplying by the sign of the permutation.  

Thus, to define the straightening, we assume that the straightening has already been constructed for dimensions less than $d$, and consider the $(d+1)$-tuples of points in $E\Gamma$, with the action of $\Gamma$ on these tuples.  For each orbit, we take one $(d+1)$-tuple in the orbit and select a $d$-simplex with those vertices, such that the boundary agrees with our lower-dimensional straightening choices, and take the signed barycentric subdivision to be our straightening.  Translating by $\Gamma$, this defines the straightening on all $(d+1)$-tuples in the same orbit.  Each tuple in the orbit is obtained by only one translation, because the assumption that $\Gamma$ is torsion-free implies that no $(d+1)$-tuple is mapped to a permutation of itself by a nontrivial element of $\Gamma$.  Thus, the straightening is well-defined, without the need for averaging and introducing non-integer coefficients.  It maps each ordered $(d+1)$-tuple of points in $E\Gamma$ to a singular chain, in such a way that transposing two vertices in the ordering multiplies the chain by $-1$.  The straightening of any non-rainbow simplex is $0$, because its two vertices of the same color become the same point of $E\Gamma$.
\end{proof}

In~\cite{loeh22}, L\"oh, Moraschini, and Sauer prove a theorem similar to Lemma~\ref{lem:thick-rainbow}, where no simplices are rainbow, but where the 1-complex of each color is assumed to generate an amenable subgroup of $\pi_1(P)$, rather than a trivial subgroup.  Using the lemma, we can finish the proof of Theorem~\ref{thm:cor-source}.

\begin{proof}[Proof of Theorem~\ref{thm:cor-source}]
The proof follows the proof of Theorem~\ref{thm:big}.  The main change is in the conclusion of Lemma~\ref{lem:triangulate-cantor}, which becomes the following: there is a cycle $z = \sum_i f_i \otimes \sigma_i$ in the homology class of $j_*^M[M]$ with a $\pi_1$-killing vertex coloring, such that
\[\sum_{\mathrm{rainbow}\ \sigma_i} \int \abs{f_i}\ d\mu = 2^n \cdot \Area_0\left(\widetilde{Z}_0\right).\]
The only change needed to prove this statement is to estimate the norms up in $X \times \widetilde{M}$ instead of mapping down to $M$.  That is, for each $n$-dimensional simplex $E$ of the triangulated polyhedron $Z_n$, instead of taking the singular simplex $\varphi\vert_E$ in $M$ with coefficient $\pm \mu(A_E)$, we take the element $\pm 1_{A_E} \otimes \widetilde{\varphi\vert_E}$, where $1_{A_E}$ denotes the characteristic function of $A_E$ in $X$.  This forms a cycle $z'$ in $C(X; \mathbb{Z}) \otimes_{\mathbb{Z}[\Gamma]} C_n(\widetilde{M}; \mathbb{Z})$, and its barycentric subdivision $z$ is the cycle we want, with $\pi_1$-killing coloring coming from coloring $Z_n$ by connected components of the various levels $Z_{i} \setminus Z_{i-1}$, as in Lemma~\ref{lem:triangulate-cantor}.

To check that $z'$ is indeed a cycle, consider any $(n-1)$-dimensional face $F$ of $Z_n$.  In $\partial z'$, the singular simplex $\widetilde{\varphi\vert_F}$ may appear several times as part of the boundaries of singular simplices $\widetilde{\varphi\vert_E}$ for various $n$-dimensional faces $E$ of $Z_n$, and we need to check that the sum of their contributions to the boundary of $z'$ is zero.  The definition of thick triangulation requires $A_E \subseteq A_F$ for every such $E$, as well as requiring that the thick faces partition $X \times \widetilde{M}$.  Thus, one side of the thick face $A_F \times \widetilde{\varphi\vert_F}(F)$ in $X \times \widetilde{M}$ is divided into thick faces $A_E \times \widetilde{\varphi\vert_E}(E)$ that each contribute $+1_{A_E} \otimes \widetilde{\varphi\vert_F}$ to $\partial z'$, and the other side is divided into thick faces that each contribute $-1_{A_E} \otimes \widetilde{\varphi\vert_F}$ to $\partial z'$.  The total sum is $(+1_{A_F} - 1_{A_F}) \otimes \widetilde{\varphi\vert_F}$, which is zero.  Checking every $(n-1)$-dimensional face $F$ of the thick triangulation $\widetilde{Z}_n$ in this way shows that $z'$ is a cycle.

Applying Lemma~\ref{lem:thick-rainbow}, we obtain the inequality
\[\Vert c_*j_*^M[M] \Vert_{\mathbb{Z}}^X \leq (n+1)!\cdot 2^n \cdot \Area_0\left(\widetilde{Z}_0\right).\]
We observe that $c_*j_*^M[M]$ and $j_*^{B\Gamma}c_*[M]$ mean the same thing: take the class on $X \times E\Gamma$ corresponding to the fundamental homology class on $M$.  Thus we have
\[\Vert j_*^{B\Gamma}c_*[M] \Vert_{\mathbb{Z}}^X \leq (n+1)!\cdot 2^n\cdot \Area_0\left(\widetilde{Z}_0\right),\]
and the remainder of the proof of Theorem~\ref{thm:big} proceeds as before.  Namely, for every $\varepsilon > 0$ we apply Lemma~\ref{lem:volball-cantor} with $R = 1$, to get a filtration where the inequality above holds; then, combining with Lemma~\ref{lem:inequality-string} and taking $\varepsilon \rightarrow 0$ gives
\[\Vert j_*^{B\Gamma}c_*[M] \Vert_{\mathbb{Z}}^X \leq (n+1)!\cdot 16^n(n!)^2\cdot V_1(\widetilde{M}) \cdot \Vol M = (n+1)16^n(n!)^3\cdot V_1(\widetilde{M}) \cdot \Vol M.\]
\end{proof}

\bibliography{reproof-bib}

\providecommand{\bysame}{\leavevmode\hbox to3em{\hrulefill}\thinspace}
\providecommand{\MR}{\relax\ifhmode\unskip\space\fi MR }
\providecommand{\MRhref}[2]{%
  \href{http://www.ams.org/mathscinet-getitem?mr=#1}{#2}
}
\providecommand{\href}[2]{#2}
\begin{thebibliography}{LLNR22}

\bibitem[AK16]{alpert16}
Hannah Alpert and Gabriel Katz, \emph{Using simplicial volume to count
  multi-tangent trajectories of traversing vector fields}, Geom. Dedicata
  \textbf{180} (2016), 323--338.

\bibitem[Bro82]{brown82}
Kenneth~S Brown, \emph{Cohomology of groups}, Graduate Texts in Mathematics,
  Springer New York, NY, 1982.

\bibitem[BS21]{braun21}
Sabine Braun and Roman Sauer, \emph{Volume and macroscopic scalar curvature},
  Geom. Funct. Anal. \textbf{31} (2021), no.~6, 1321--1376.

\bibitem[BZ88]{burago88}
Yu.~D. Burago and V.~A. Zalgaller, \emph{Geometric inequalities}, Grundlehren
  der mathematischen Wissenschaften [Fundamental Principles of Mathematical
  Sciences], vol. 285, Springer-Verlag, Berlin, 1988, Translated from the
  Russian by A. B. Sosinski\u{\i}, Springer Series in Soviet Mathematics.

\bibitem[CS19]{campagnolo19}
Caterina Campagnolo and Roman Sauer, \emph{Counting maximally broken {M}orse
  trajectories on aspherical manifolds}, Geom. Dedicata \textbf{202} (2019),
  387--399.

\bibitem[Ele21]{elek21}
G\'{a}bor Elek, \emph{Free minimal actions of countable groups with invariant
  probability measures}, Ergodic Theory Dynam. Systems \textbf{41} (2021),
  no.~5, 1369--1389.

\bibitem[FLPS16]{frigerio16}
Roberto Frigerio, Clara L\"{o}h, Cristina Pagliantini, and Roman Sauer,
  \emph{Integral foliated simplicial volume of aspherical manifolds}, Israel J.
  Math. \textbf{216} (2016), no.~2, 707--751.

\bibitem[GM88]{goresky88}
Mark Goresky and Robert MacPherson, \emph{Stratified {M}orse theory},
  Ergebnisse der Mathematik und ihrer Grenzgebiete (3) [Results in Mathematics
  and Related Areas (3)], vol.~14, Springer-Verlag, Berlin, 1988.

\bibitem[Gor78]{goresky78}
R.~Mark Goresky, \emph{Triangulation of stratified objects}, Proc. Amer. Math.
  Soc. \textbf{72} (1978), no.~1, 193--200.

\bibitem[Gro82]{gromov82}
Michael Gromov, \emph{Volume and bounded cohomology}, Inst. Hautes \'{E}tudes
  Sci. Publ. Math. (1982), no.~56, 5--99 (1983).

\bibitem[Gro99]{gromov99}
Misha Gromov, \emph{Metric structures for {R}iemannian and non-{R}iemannian
  spaces}, Progress in Mathematics, vol. 152, Birkh\"{a}user Boston, Inc.,
  Boston, MA, 1999, Based on the 1981 French original [ MR0682063 (85e:53051)],
  With appendices by M. Katz, P. Pansu and S. Semmes, Translated from the
  French by Sean Michael Bates.

\bibitem[Gro09]{gromov09}
Mikhail Gromov, \emph{Singularities, expanders and topology of maps. {I}.
  {H}omology versus volume in the spaces of cycles}, Geom. Funct. Anal.
  \textbf{19} (2009), no.~3, 743--841.

\bibitem[Gut10]{guth10}
Larry Guth, \emph{Systolic inequalities and minimal hypersurfaces}, Geom.
  Funct. Anal. \textbf{19} (2010), no.~6, 1688--1692.

\bibitem[Gut11]{guth11}
\bysame, \emph{Volumes of balls in large {R}iemannian manifolds}, Ann. of Math.
  (2) \textbf{173} (2011), no.~1, 51--76.

\bibitem[Gut17]{guth17}
\bysame, \emph{Volumes of balls in {R}iemannian manifolds and {U}ryson width},
  J. Topol. Anal. \textbf{9} (2017), no.~2, 195--219.

\bibitem[HM06]{hjorth06}
Greg Hjorth and Mats Molberg, \emph{Free continuous actions on zero-dimensional
  spaces}, Topology Appl. \textbf{153} (2006), no.~7, 1116--1131.

\bibitem[Iva87]{ivanov87}
N.~V. Ivanov, \emph{Foundations of the theory of bounded cohomology}, Journal
  of Soviet Mathematics \textbf{37} (1987), no.~3, 1090--1115.

\bibitem[LLNR22]{liokumovich2022filling}
Yevgeny Liokumovich, Boris Lishak, Alexander Nabutovsky, and Regina Rotman,
  \emph{Filling metric spaces}, Duke Mathematical Journal \textbf{1} (2022),
  no.~1, 1--38.

\bibitem[LMS22]{loeh22}
Clara L\"{o}h, Marco Moraschini, and Roman Sauer, \emph{Amenable covers and
  integral foliated simplicial volume}, New York J. Math. \textbf{28} (2022),
  1112--1136.

\bibitem[L{\"{u}}c98]{lueck98}
Wolfgang L{\"{u}}ck, \emph{Dimension theory of arbitrary modules over finite
  von {N}eumann algebras and {$L^2$}-{B}etti numbers. {I}. {F}oundations}, J.
  Reine Angew. Math. \textbf{495} (1998), 135--162.

\bibitem[L{\"{u}}c02]{lueck02}
\bysame, \emph{{$L^2$}-invariants: theory and applications to geometry and
  {$K$}-theory}, Ergebnisse der Mathematik und ihrer Grenzgebiete. 3. Folge. A
  Series of Modern Surveys in Mathematics [Results in Mathematics and Related
  Areas. 3rd Series. A Series of Modern Surveys in Mathematics], vol.~44,
  Springer-Verlag, Berlin, 2002.

\bibitem[Nab22]{nabutovsky19}
Alexander Nabutovsky, \emph{Linear bounds for constants in {G}romov's systolic
  inequality and related results}, Geom. Topol. \textbf{26} (2022), no.~7,
  3123--3142.

\bibitem[Pap20]{papasoglu20}
Panos Papasoglu, \emph{Uryson width and volume}, Geom. Funct. Anal. \textbf{30}
  (2020), no.~2, 574--587.

\bibitem[Sab22]{sabourau22}
St\'{e}phane Sabourau, \emph{Macroscopic scalar curvature and local
  collapsing}, Ann. Sci. \'{E}c. Norm. Sup\'{e}r. (4) \textbf{55} (2022),
  no.~4, 919--936.

\bibitem[Sch05]{schmidt05}
Marco Schmidt, \emph{{$L^2$}-betti numbers of {$\mathcal{R}$}-spaces and the
  integral foliated simplicial volume}, Ph.D. thesis, WWU M\"unster, 2005.

\bibitem[SY78]{schoen78}
Richard Schoen and Shing~Tung Yau, \emph{Incompressible minimal surfaces,
  three-dimensional manifolds with nonnegative scalar curvature, and the
  positive mass conjecture in general relativity}, Proc. Nat. Acad. Sci. U.S.A.
  \textbf{75} (1978), no.~6, 2567.

\bibitem[SY79]{schoen79}
R.~Schoen and S.~T. Yau, \emph{On the structure of manifolds with positive
  scalar curvature}, Manuscripta Math. \textbf{28} (1979), no.~1-3, 159--183.

\end{thebibliography}
\bibliographystyle{amsalpha}
\end{document}